\documentclass[12pt,reqno]{amsart}
\usepackage{amssymb,amscd}
\usepackage[alphabetic]{amsrefs}
\usepackage[all]{xy}
\usepackage[headings]{fullpage}
\usepackage[T1]{fontenc}
\usepackage{libertine}
\usepackage{stackengine}

\numberwithin{equation}{section}

\swapnumbers
\theoremstyle{plain}
\newtheorem{thm}[subsection]{Theorem}

\newtheorem{prop}[subsection]{Proposition}
\newtheorem{propss}[subsubsection]{Proposition}
\newtheorem{lemma}[subsubsection]{Lemma}
\newtheorem{cor}[subsection]{Corollary}
\newtheorem{corss}[subsubsection]{Corollary}

\theoremstyle{definition}

\theoremstyle{remark}
\newtheorem{rem}[subsection]{Remark}
\newtheorem{rems}[subsection]{Remarks}
\newtheorem{remss}[subsubsection]{Remark}

\numberwithin{equation}{section}
\theoremstyle{plain}
\newtheorem{prop-def}[subsection]{Proposition/Definition}

\usepackage[OT2,T1]{fontenc}
\DeclareSymbolFont{cyrletters}{OT2}{wncyr}{m}{n}
\DeclareMathSymbol{\sha}{\mathalpha}{cyrletters}{"58}

\renewcommand{\AA}{\mathcal{A}}

\newcommand{\CC}{\mathcal{C}}

\newcommand{\DD}{\mathcal{D}}

\newcommand{\JJ}{\mathcal{J}}

\renewcommand{\SS}{\mathcal{S}}

\newcommand{\XX}{\mathcal{X}}

\newcommand{\YY}{\mathcal{Y}}

\newcommand{\Curve}{\mathcal{C}}

\newcommand{\EE}{\mathcal{E}}

\newcommand{\LL}{\mathcal{L}}

\newcommand{\OO}{\mathcal{O}}

\newcommand{\F}{\mathbb{F}}
\newcommand{\Fp}{{\mathbb{F}_p}}

\newcommand{\Fq}{{\mathbb{F}_q}}
\newcommand{\Fqn}{{\mathbb{F}_{q^n}}}

\newcommand{\Fpbar}{{\overline{\mathbb{F}}_p}}

\newcommand{\Fqbar}{{\overline{\mathbb{F}}_q}}

\newcommand{\kbar}{{\overline{k}}}

\newcommand{\ratto}{{\dashrightarrow}}
\newcommand{\Ql}{{\mathbb{Q}_\ell}}

\newcommand{\Qp}{{\mathbb{Q}_p}}

\newcommand{\Z}{\mathbb{Z}}

\newcommand{\Q}{\mathbb{Q}}

\newcommand{\A}{\mathbb{A}}
\renewcommand{\P}{\mathbb{P}}

\newcommand{\m}{\mathfrak{m}}

\newcommand{\<}{\langle}
\renewcommand{\>}{\rangle}
\newcommand{\into}{\hookrightarrow}

\newcommand{\isoto}{\,\tilde{\to}\,}

\newcommand{\tensor}{\otimes}
\newcommand{\compose}{\circ}
\newcommand{\sdp}{{\rtimes}}
\newcommand{\nodiv}{\not|}
\def\nodiv{\mathrel{\mathchoice{\not|}{\not|}{\kern-.2em\not\kern.2em|}
{\kern-.2em\not\kern.2em|}}}
\newcommand{\PGL}{\mathrm{PGL}}

\newcommand{\GL}{\mathrm{GL}}

\newcommand{\G}{\mathbb{G}}

\DeclareMathOperator{\tr}{Tr}

\DeclareMathOperator{\res}{Res}

\DeclareMathOperator{\ord}{ord}
\DeclareMathOperator{\rk}{Rank}

\DeclareMathOperator{\Hom}{Hom}
\DeclareMathOperator{\aut}{Aut}
\DeclareMathOperator{\Pic}{Pic}

\DeclareMathOperator{\NS}{NS}

\DeclareMathOperator{\Reg}{Reg}

\DeclareMathOperator{\Br}{Br}

\DeclareMathOperator{\lcm}{lcm}
\DeclareMathOperator{\gal}{Gal}

\DeclareMathOperator{\Sel}{Sel}
\DeclareMathOperator{\Fr}{Fr}

\def\clap#1{\hbox to 0pt{\hss#1\hss}}

\DeclareMathOperator{\BS}{BS}
\newcommand{\Hbar}{\overline{H}}
\newcommand{\Gd}{\underline{G}_d}

\begin{document}
\title[Brauer-Siegel]{On the Brauer-Siegel ratio\\
for abelian varieties over function fields}

\author{Douglas Ulmer}
\address{Department of Mathematics \\ University of Arizona
  \\ Tucson, AZ~~85721 USA}
\email{ulmer@math.arizona.edu}

\date{\today}

%\subjclass[2010]{Primary 14J27; Secondary 14G99, 11G99}

\begin{abstract}
  Hindry has proposed an analogue of the classical Brauer-Siegel
  theorem for abelian varieties over global fields.  Roughly speaking,
  it says that the product of the regulator of the Mordell-Weil group
  and the order of the Tate-Shafarevich group should have size comparable
  to the exponential differential height.  Hindry-Pacheco and Griffon
  have proved this for certain families of elliptic curves over
  function fields using analytic techniques.  Our goal in this work is
  to prove similar results by more algebraic arguments, namely by a
  direct approach to the Tate-Shafarevich group and the regulator.  We
  recover the results of Hindry-Pacheco and Griffon and extend them to
  new families, including families of higher-dimensional abelian
  varieties.
\end{abstract}

\maketitle

\section{Introduction}
The classical Brauer-Siegel theorem \cite{Brauer50} says that if $K$
runs through a sequence of Galois extensions of $\Q$ with
discriminants $d=d_K$ satisfying $[K:\Q]/\log d\to0$, then
$$\log(Rh)/\log\sqrt{d}\to1$$ 
where $R=R_K$ and $h=h_K$ are the regulator and class number of $K$.
The proof uses the class number formula
$$\res_{s=1}\zeta_K(s)=\frac{2^{r_1}(2\pi)^{r_2}Rh}{w\sqrt{d}}$$
and analytic methods.

In \cite{Hindry07}, Hindry conjectured an analogue of the Brauer-Siegel
theorem for abelian varieties.  If $A$ is an abelian variety over a
global field $K$ with regulator $R$, Tate-Shafarevich group $\sha$
(assumed to be finite), and exponential differential height $H$ (definitions
below), Hindry proposed that the Brauer-Siegel ratio
$$\BS(A):=\log(R|\sha|)/\log(H)$$
should tend to 1 for any sequence of abelian varieties over a fixed
$K$ with $H\to\infty$.

In \cite{HindryPacheco16}, Hindry and Pacheco considered the case
where $K$ is a global function field of characteristic $p>0$.  Assuming
the finiteness of $\sha$, they proved (Cor.~1.13) that
\begin{equation}\label{eq:BS-HP}
0\le\liminf_A\BS(A)\le\limsup_A\BS(A)=1
\end{equation}
where the limits are over the family of all non-constant abelian
varieties of a fixed dimension over $K$ ordered by height.  Note that
this leaves open the possibility of a sequence of abelian varieties
with Brauer-Siegel ratio tending to a limit $<1$, a possibility not
envisioned in Hindry's earlier paper.  Hindry and Pacheco also
conjectured and gave evidence for the claim that the lower bound
$0\le\liminf_A\BS(A)$ should be an equality when $A$ runs through the
family of quadratic twists of a fixed elliptic curve.  Moreover, they
gave an example (Thm.~1.4) of a family of elliptic curves $E$ with
$H\to\infty$ and proved $\lim_E\BS(E)=1$ without having to assume any
unproven conjectures.  In his Paris VII thesis, Griffon gave several
other examples of families of elliptic curves where $\lim_E\BS(E)=1$
again without assuming unproven conjectures.

As with the original Brauer-Siegel theorem, the analyses of
Hindry-Pacheco and Griffon use analytic techniques.
More precisely, finiteness of the Tate-Shafarevich group
implies the conjecture of Birch and Swinnerton-Dyer (in its strong
form), and so a class number formula of the shape
$$L^*(A)=\alpha\frac{|\sha|R}{H}$$
where $L^*(A)$ is the leading Taylor coefficient of the $L$-function
at $s=1$ and $\alpha$ is a relatively innocuous, non-zero factor.  (We
will give the precise statement below.)  Hindry-Pacheco and
Griffon then prove their results by estimating (and in some cases
calculating quite explicitly) $L^*(A)$.

Our goal in this work is to prove several results about Brauer-Siegel
ratios by more algebraic arguments, in other words through a direct
approach to the Tate-Shafarevich group and the regulator.  More
precisely, we prove the following results without recourse to
$L$-functions: 
\begin{enumerate}
\item a transparent and conceptual proof that $\liminf_A\BS(A)\ge0$
  via a lower bound on the regulator;
\item a new connection between the growth of $|\sha|$ as the finite
  ground field is extended and the number $R|\sha|$ over the given field; 
\item a general calculation of the limiting Brauer-Siegel ratio for
  the sequence $E^{(p^n)}$ of Frobenius pull-backs of an elliptic
  curve $E$;
\item a new proof that $\lim_d\BS(E_d)=1$ in the families of elliptic
  curves studied by Hindry-Pacheco and Griffon;
\item proofs that $\lim_d\BS(J_d)=1$ for families of Jacobians of all
  dimensions;
\item and results on quadratic twists that illustrate the limitations of
  our $p$-adic techniques.
\end{enumerate}

``Without recourse to $L$-functions'' means by algebraic methods.  We
do use the BSD formula, but this is just a bookkeeping device for the
connections between cohomology and other invariants.  We do not use
the Euler product or any properties of $L(A,s)$ as a function of
$s$. That said, we have not eliminated analysis entirely:  points (4-6)
above all require an equidistibution result for the action of
multiplication by $p$ on $\Z/d\Z$.

The plan of the paper is as follows: In Section~2, we set up notation,
review and extend certain auxiliary results of Hindry-Pacheco on
component groups, and prove a lemma useful for estimating heights.
In Section~3, we prove a general integrality result on regulators of
abelian varieties which leads immediately to a lower bound on the
Brauer-Siegel ratio.  In Section~4, we introduce ``$\dim\sha(A)"$, a
new and extremely useful technical device which is closely related to
slopes of $L$-functions and which is computable in many interesting
situations.  As a first application, in Section~5 we compute the
limiting Brauer-Siegel ratio for the sequence of Frobenius pull-backs
of an elliptic curve.  Sections~6 through 9 develop $p$-adic cohomological
machinery that allows one to compute $\dim\sha(A)$ and estimate
$\BS(A)$ for Jacobians of curves with N\'eron models related to
products of Fermat curves.  In the rest of the paper, we use this
machinery to recover the results of Hindry-Pacheco and Griffon and to
extend them to higher genus Jacobians.  Section~10 discusses curves
defined by equations involving 4 monomials.  Section~11 discusses
curves coming from Berger's construction \cite{Berger08}.  Finally, in
Section~12 we consider twists of constant elliptic curves.

It is a pleasure to thank Richard Griffon for several helpful
comments and an anonymous referee for his or her careful reading of
the paper and valuable suggestions.

\numberwithin{equation}{subsection}
\section{Preliminaries}\label{s:prelims}
\subsection{Notation and definitions}\label{ss:prelims}
We set notation and recall definitions which will be used throughout
the paper.

Fix a prime number $p$, a power $q$ of $p$, and a smooth, projective,
absolutely irreducible curve $\CC$ of genus $g_\CC$ over $k=\Fq$, the
field of $q$ elements.  Let $K$ be the function field $\Fq(\CC)$.  We
write $v$ for a place of $K$, $d_v$ for the degree of $v$, $K_v$ for
the completion of $K$ at $v$, $\OO_v$ for the ring if integers in
$K_v$, and $k_v$ for the residue field, a finite extension of $k$ of
degree $d_v$

Let $A$ be an abelian variety %of dimension $\dim(A)$
over $K$ with dual $\hat A$.  A theorem of Lang and N\'eron guarantees
that the \emph{Mordell-Weil groups} $A(K)$ and $\hat A(K)$ are
finitely generated abelian groups.  (See \cite{LangNeron59}, or
\cite{Conrad06} for a more modern account.)

There is a bilinear pairing
$$\<\cdot,\cdot\>:A(K)\times\hat A(K)\to\Q$$
which is non-degenerate modulo torsion.  (This is the canonical
N\'eron-Tate height divided by $\log q$.  See \cite{Neron65} for the
definition and \cite[B.5]{HindrySilvermanDG} for a friendly
introduction.)  Choosing a basis $P_1,\dots,P_r$ for $A(K)$ modulo
torsion and a basis $\hat P_1,\dots,\hat P_r$ for $\hat A(K)$ modulo
torsion, we define the \emph{regulator} of $A$ as
$$\Reg(A):=|\det\< P_i,\hat P_j\>_{1\le i,j\le r}|.$$
The regulator is a positive rational number, well-defined independently
of the choice of bases.

We write $H^1(K,A)$ for the \'etale cohomology of $K$ with
coefficients in $A$ and similarly for $H^1(K_v,A)$.  The
\emph{Tate-Shafarevich group} of $A$  is defined as
$$\sha(A):=\ker\left(H^1(K,A)\to\prod_v H^1(K_v,A)\right)$$
where the product of over the places of $K$ and the map is the product
of the restriction maps.  This group is conjectured to be finite, and
we assume this conjecture throughout the paper.  However, in all of
the explicit calculations below, we can in fact prove that $\sha(A)$
is finite without additional assumptions.

Let $\AA\to\CC$ be the N\'eron model of $A/K$.  This is a smooth group
scheme over $\CC$ with a certain universal property whose generic
fiber is $A/K$.  See \cite{BoschLutkebohmertRaynaudNM} for a modern
account.  Let $s:\CC\to\AA$ be the zero-section.  We define an
invertible sheaf $\omega$ on $\CC$ by
$$\omega:=s^*\left(\Omega^{\dim(A)}_{\AA/\CC}\right)
=\bigwedge^{\dim(A)}s^*\left(\Omega^1_{\AA/\CC}\right).$$
The \emph{exponential differential height} of $A$ (which we often
refer to simply as the \emph{height}) is
$$H(A):=q^{\deg\omega}.$$
If $A$ is an elliptic curve and $\CC=\P^1$, then $\deg\omega$ has
simple interpretation in terms of the degrees of the coefficients in a
Weierstrass equation defining $A$.  See \cite[Lecture~3]{Ulmer11} for
details.

For each place $v$ of $K$, we write $c_v$ for the number of connected
components of the special fiber of $\AA$ at $v$ which are defined over
the residue field.  We define the \emph{Tamagawa number} of $A$ as
$$\tau(A):=\prod_vc_v.$$
(This usage is in conflict with our earlier papers, in particular
\cite{Ulmer14b}, where the Tamagawa number is defined to be
$$\frac{\tau(A)}{H(A)q^{\dim(A)(g_\CC-1)}}.$$  
The earlier usage is historically more appropriate, as the definition
there is a volume defined in close analogy with Tamagawa's work on
linear algebraic groups, cf.~\cite{WeilAAG}, but the terminology we
adopt here is more convenient for our current purposes.)

Next we consider the Hasse-Weil $L$-function of $A$ over $K$, denoted
$L(A,s)$.  It is a function of a complex variable $s$ defined as an
Euler product over the places of $K$ which is convergent in the
half-plane $\Re s>3/2$ and which is known to have a meromorphic
continuation to the whole $s$-plane.  More precisely, $L(A,s)$ is a
rational function in $q^{-s}$, and if the $K/k$-trace of $A$ is
trivial, then $L(A,s)$ is in fact a polynomial in $q^{-s}$ of the form
$$\prod_i(1-\alpha_iq^{-s})$$
where the inverse root $\alpha_i$ are Weil integers of size $q$.

We define the leading coefficient of the $L$-function as
$$L^*(A):=\frac{1}{(\log q)^r}\frac1{r!}
\left.\left(\frac{d}{ds}\right)^rL(A,s)\right|_{s=1}$$
where $r$ is the order of vanishing $r:=\ord_{s=1}L(A,s)$.  (With the
factor $1/(\log q)^r$, this is the leading coefficient of $L$ as a
rational function in $T=q^{-s}$, and with this normalization, it has
the virtue of being a rational number.)

All of the invariants mentioned above are connected by the conjecture
of Birch and Swinnerton-Dyer (``BSD conjecture''), which we take to
be the conjunction of the following three statements:
\begin{enumerate}
\item $\ord_{s=1}L(A,s)=\rk A(K)$
\item $\sha(A)$ is finite (with order denoted $|\sha(A)|$)
\item we have an equality
$$L^*(A)=\frac{\Reg(A)|\sha(A)|}{H(A)}
\frac{\tau(A)}{q^{\dim(A)(g_\CC-1)}|A(K)_{tor}|\cdot|\hat A(K)_{tor}|}$$
\end{enumerate}
It is known that parts (1) and (2) are equivalent, and when they hold,
part (3) holds as well.  (See \cite{KatoTrihan03} for the end of a
long line of reasoning leading to these results.)

From the point of view of the Brauer-Siegel ratio, the main terms of
interest in the third part of the BSD conjecture are $\Reg(A)$,
$|\sha(A)|$, and $H(A)$, whereas the other factors are either constant
($q^{\dim(A)(g_\CC-1)}$) or turn out to be negligible ($\tau(A)$ and
$|A(K)_{tor}\times \hat A(K)_{tor}|$).  We will discuss the Tamagawa
number and the results of Hindry and Pacheco on it in the next
section, whereas the torsion subgroups $A(K)_{tor}$ and
  $\hat A(K)_{tor}$ will play almost no role in our analysis.

% ?Not clear we need torsion.  If so:  In elliptic case, we have a
% uniform bound depending only on $g_\CC$.  It's an old result,
% rediscovered many times.  See PCMI Prop. 7.1 for a sketch and
% references, including to a strong bound in terms of the gonality of
% $\CC$ due to Poonen.  Uniform bounds are expected (in terms of $d$ and
% $g_\CC$) for abelian varieties, but these are not known.
% Hindry-Pacheco (Thm.~3.8) show that
% $$\left|A(K)_{tor}\times\hat A(K)_{tor}\right|=O(H(A)^\epsilon)$$
% for all $\epsilon>0$.

\subsection{Bounds on Tamagawa numbers (1)}
In \cite[Prop.~6.8]{HindryPacheco16}, Hindry and Pacheco bound the
Tamagawa number in terms of the height under certain tameness
assumptions.  More precisely, they showed that for a fixed global
field $K$, as $A$ varies over all abelian varieties of fixed dimension
$d$ over $K$, we have
\begin{equation*}%\label{eq:HP-Tamagawa-bound}
\tau(A)=O(H^\epsilon)
\end{equation*}
for all $\epsilon>0$, provided that $p>2\dim(A)+1$ or $A$ 
has everywhere semi-stable reduction.

In this section and Sections~\ref{ss:T2} and \ref{ss:T3}, we outline
three improvements of this result, all motivated by applications later
in the paper.

\begin{lemma}\label{lemma:tau-ec}
Let $E$ run through the set of all elliptic curves over a global
function field $K$.  Then
$$\tau(E)=O(H(E)^\epsilon)$$
for every $\epsilon>0$.
\end{lemma}

The point is that we allow arbitrary characteristic and make no
semi-stability hypothesis.  This result was also proven by Griffon
\cite[Thm.~1.5.4]{GriffonThesis}, but we include a proof here for the
convenience of the reader.

\begin{proof}
This follows easily from Ogg's formula \cite{Ogg67} (see also
\cite{Saito88} for a more general result proven with modern methods).
Indeed, if $\Delta_v$ is a minimal discriminant for $E$ at the place
$v$, Ogg's formula says that
$$\ord_v(\Delta_v)=c_v+f_v-1$$
where $f_v$ is the exponent of the conductor of $E$ at $v$.  Summing
over places where $E$ has bad reduction (i.e., where
$\ord_v(\Delta_v)\ge1$) and using that $f_v-1\ge0$ at these places, we
have
$$\sum_vc_vd_v\le\sum_v\ord_v(\Delta_v)d_v\le12\deg(\omega)$$
where $d_v$ is the degree of $v$ and where the last inequality holds
because $\Delta$ can be interpreted as a section of
$\omega^{\tensor12}$.  This recovers the main bound (Theorem~6.5 of
\cite{HindryPacheco16}), and the rest of the argument---converting this
additive bound to a multiplicative bound---proceeds
exactly as in \cite[Prop.~6.8]{HindryPacheco16}.
\end{proof}

\subsection{Families from towers of fields}\label{ss:towers}
Let $A$ be an abelian variety over a function field $K$.  For each
positive integer $d$ (or positive integer $d$ prime to $p$), let $K_d$
be a geometric extension of $K$, and let $A_d=A\times_KK_d$.  This
gives a sequence of abelian varieties and one may ask about the
behaviour of $\BS(A_d)$ as $d\to\infty$.

For most of the paper, we will be concerned with the special case
where there are isomorphisms $K_d\cong K$ for all $d$.  In this case,
we may view the sequence $A_d$ as a sequence of abelian varieties over
a \emph{fixed} function field.  This is the context of the results and
conjectures of Hinry and Pacheco, and we will give four examples in
the rest of this section.  Nevertheless, the general case is also
interesting, and we will give develop foundational results in a more
general context in Section~\ref{ss:TGGE}.

\subsubsection{Kummer families}
Let $K=\Fq(t)$, and for each positive integer $d$ prime to $p$, let
$K_d=\Fq(u)$ where $u^d=t$.  Note that the extension $K_d/K$ is
unramified away from the places $t=0$ and $t=\infty$ of $K$.  Let $A$ be
an abelian variety over $K$, and let $A_d$ be the abelian variety over
$K$ obtained by base change to $K_d$, followed by the isomorphism of
fields $\Fq(u)\cong\Fq(t)$, $u\mapsto t$.  (In more vivid terms, $A_d$
is the result of substituting $t^d$ for each appearance of $t$ in the
equations defining $A$.)  We say that the sequence of abelian
varieties $A_d$ is the \emph{family associated to $A$ and the Kummer
  tower\/}.  Such families have been a prime source of examples for
the Brauer-Siegel ratio.

\subsubsection{Artin-Schreier families}
We may proceed analogously with the tower of Artin-Schreier
extensions.  Again, let $K=\Fq(t)$, and for each positive integer $d$,
let $K_d=\Fq(u)$ where $u^{p^d}-u=t$.  Note that the extension $K_d/K$
is unramified away from the place $t=\infty$ of $K$.  Let $A$ be an
abelian variety over $K$, and let $A_d$ be the abelian variety over
$K$ obtained by base change to $K_d$ followed by the isomorphism of
fields $\Fq(u)\cong\Fq(t)$, $u\mapsto t$.  (In more vivid terms, $A_d$
is the result of substituting $t^{p^d}-t$ for each appearance of $t$
in the equations defining $A$.)  We say that the sequence of abelian
varieties $A_d$ is the \emph{family associated to $A$ and the
  Artin-Schreier tower\/}.

\subsubsection{Division tower families}
One may also consider an elliptic curve variant: Let $K$ be the
function field $\Fq(E)$ where $E$ is an elliptic curve over $\Fq$.
For each positive integer $d$ prime to $p$, consider the field
extension $K_d/K$ associated to the multiplication map $d:E\to E$.
Thus $[K_d:K]=d^2$, but $K_d$ is canonically isomorphic as a field
(even as an $\Fq$-algebra) to $K$.  Given an abelian variety $A$ over
$K$, let $A_d$ be the abelian variety over $K$ obtained by
base-changing $A$ to $K_d$ and then using the isomorphism of fields
$K_d\cong K$.  We say that the sequence $A_d$ of abelian varieties
over $K$ is the \emph{family associated to a division tower\/}.
Everything we say about Kummer and Artin-Schreier towers has an
obvious analogue for division towers.  In most cases the latter is
simpler because in the division case, $K_d/K$ is unramified.

\subsubsection{$\PGL_2$ families}
Let $K=\Fq(t)$ and for each positive integer $d$ let $K_d=\Fq(u)$
where $\Fq(u)/\Fq(t)$ is the field extension associated to the
quotient morphism 
$$\P^1\to\P^1/\PGL_2(\F_{p^d})\cong \P^1.$$  
We normalize the isomorphism so that the $\F_{p^d}$-rational points on
the upper $\P^1$ map to $0$ and
$\P^1(\F_{p^{2d}})\setminus\P^1(\F_{p^d})$ maps to $1$.  Then the
extension $K_d/K$ is unramified away from the places $t=0$ and $t=1$
of $K$, and it is tamely ramified over $t=1$.  Given an abelian
variety $A$ over $K$, let $A_d$ be the abelian variety over $K$
obtained by base-changing $A$ to $K_d$ and then using the isomorphism
of fields $\Fq(u)\cong\Fq(t)$, $u\mapsto t$.  We say that the sequence
$A_d$ of abelian varieties over $K$ is the \emph{family associated to
  the $\PGL_2$ tower\/}.

\medskip
The discussion above gives four different meanings to the notations
$K_d$ and $A_d$!  Which meaning is intended in each use below should
be clear from the context.

We end this section with a simple lemma that plays a key role in our
analysis of Tamagawa numbers in families associated to towers.

\begin{lemma}\label{lemma:easy-place-bounds}
  Let $K=\Fq(\Curve)$ be a function field, and let $K_d$ be a sequence
  of geometric extensions of $K$ such that the genus of \textup{(}the
  curve associated to\textup{)} $K_d$ is $\le1$ for all $d$.  Then for
  every place $v$ of $K$, there is a constant $C_v$ depending only on
  $q$ and $\deg v$ such that for all $d$, the number of places of
  $K_d$ dividing $v$ is at most $C_v[K_d:K]/\log[K_d:K]$.
\end{lemma}

\begin{proof}
  Write $D=[K_d:K]$ and set $x=\log D/\log q$.  Fix a place $v$ of $K$.
  Then the number of places $w$ of $K_d$ dividing $v$ and of absolute
  degree $\ge x$ is at most 
$$\frac D{x/\deg v}=\deg v\log q\frac D{\log D}.$$    
On the other hand, by the Weil bound, the total number of places of
$K_d$ of degree $\le x$ is bounded by $Cq^x/x=C'D/\log D$ where $C$
and $C'$ depend only on $q$, $\deg v$ and the genus of $K_d$.  Since
the latter is either 0 or 1, the constant can be taken to depend only
on $q$ and $\deg v$.  This shows that the total number of places of
$K_d$ dividing $v$ is $\le C_v D/\log D$ where $C_v$ depends only on
$q$ and $\deg v$.
\end{proof}

\subsection{Towers of geometrically Galois extensions}\label{ss:TGGE}
In this section, we discuss a more general class of towers of fields
$K_d$ where we are able to bound Tamagawa numbers of the associated
sequences of abelian varieties.  This additional generality was
suggested by the anonymous referee, to whom we are grateful.  Readers
who are mainly interested in the applications to the Kummer tower
later in the paper are invited to skip ahead to Section~\ref{ss:T2}

\subsubsection{Geometrically Galois extensions}
Let $k$ be a field and let $K=k(\Curve)$ be the function field of a
smooth, projective, geometrically irreducible curve over $k$.  We say
that a finite, geometric extension $K_d/K$ is \emph{geometrically
  Galois} if the Galois closure $L_d$ of $K_d$ over $K$ has the form
$L_d=k_dK_d$ where $k_d$ is a finite Galois extension of $k$.
Equivalently, there is a finite Galois extension $k_d$ of $k$ such
that $k_dK_d$ is Galois over $k_dK$.  (We take $k_d$ to be minimal
such extension.)  Let $G_d=\gal(L_d/k_dK)$ and
$\Gamma_d=\gal(k_d/k)\cong\gal(k_dK/K)\cong\gal(L_d/K_d)$, so that
$\Gamma_d$ acts on $G_d$ by conjugation and $\gal(L_d/K)$ is the
semi-direct product $G_d\sdp\Gamma_d$.  We call $G_d$, with its action
of $\Gamma_d$, the \emph{geometric Galois group} of $K_d/K$ and we
call $k_d$ the \emph{splitting field} of $G_d$.  (We remark that there
is a finite \'etale group scheme $\Gd$ over $k$ attached to $G_d$ with
its $\Gamma_d$ action, and $\Gd$ becomes a constant group over $k_d$,
see \cite[\S II.1]{MilneEC}.)

\subsubsection{Towers of geometrically Galois extensions}
We now consider a tower of geometrically Galois extensions $K_d/K$
indexed by positive integers $d$ (or positive integers relatively
prime to $p$) with containments $K_d\subset K_{d'}$ whenever $d$
divides $d'$.  These containments induce surjections $G_{d'}\to G_d$
and $\Gamma_{d'}\to\Gamma_d$ which are compatible in the obvious sense
with the actions of $\Gamma_d$ and $\Gamma_{d'}$ on $G_d$ and $G_{d'}$
respectively.

Each of the families of towers in Section~\ref{ss:towers} gives an
example of a tower of geometrically Galois extensions. 

In the case of the Kummer tower, the geometric Galois group is
$G_d=\mu_d(\Fqbar)$, the splitting field $k_d$ is $\Fq(\mu_d)$, and
$\Gamma_d=\gal(\Fq(\mu_d)/\Fq)$ is the subgroup of $(\Z/d\Z)^\times$
generated by $q$.

In the Artin-Schreier tower, the geometric Galois group is
$G_d=\F_{p^d}$, the splitting field $k_d$ is $\Fq\F_{p^d}$, and
$\Gamma_d=\gal(\Fq\F_{p^d}/\Fq)$ is the cyclic group generated by the
$q$-power Frobenius.

In the division tower corresponding to an elliptic curve $E$ over
$\Fq$, the geometric Galois group is $E[d]$, the splitting field $k_d$
is $\Fq(E[d])$, and $\Gamma_d=\gal(k_d/\Fq)$ is the cyclic group
generated by the action of the $q$ power Frobenius on the $d$ torsion
points. 

In the $\PGL_2$ tower, the geometric Galois group is
$G_d=\PGL_2(\F_{p^d})$, the splitting field $k_d$ is $\Fq\F_{p^d}$,
and $\Gamma_d=\gal(\Fq\F_{p^d}/\Fq)$ is the cyclic group generated by
the $q$-power Frobenius.

For a more general class of examples, let $K_d/K$ be any of the towers
above, and fix an extension $F/K$ which is linearly disjoint from each
$K_d$ over $K$.  Then the fields $F_d:=FK_d$ form a tower of
geometrically Galois extensions with the geometric Galois group of
$F_d/F$ isomorphic to that of $K_d/K$.  Note however, that in general
the genus of $F_d$ tends to infinity with $d$.

We next consider two group-theoretic results related to these towers,
both concerning the number of orbits of $\Gamma_d$ acting on $G_d$.  (As
motivation, we note that the orbits of $\Gamma_d$ on $G_d$ are in
bijection with the closed points of the scheme $\Gd$.)

To state the first result, we make a somewhat elaborate hypothesis
on the system of groups $G_d$ with their $\Gamma_d$ actions.  
\theoremstyle{definition}
\newtheorem{hyp}[subsubsection]{Hypothesis}

\begin{hyp}\label{hyp:big-orbits}
\mbox{}
\begin{enumerate}
\item There exists a function $\phi$ of positive integers such that
    $|G_d|=\sum_{e|d}\phi(e)$ for all $d$.
\item There a decomposition $G_d=\cup_{e|d} G_{d,e}'$ such that
  $|G_{d,e}'|=\phi(e)$.
\item The action of $\Gamma_d$ on $G_d$ respects the decomposition
  above, and the orbits of $\Gamma_d$ on $G_{d,e}'$ have cardinality
  $\ge C\log|G_e|$ for some constant $C$ independent of $d$ and $e$.
  \end{enumerate}
\end{hyp}

This hypothesis clearly implies that the splitting field $k_d$ has
degree $[k_d:k]=|\Gamma_d|\ge C\log|G_d|$.  It would be interesting to
know whether the converse holds.

\begin{lemma}\label{lemma:H-examples}
Hypothesis~\ref{hyp:big-orbits} is satisfied by the Kummer,
Artin-Schreier, division, and $\PGL_2$ towers.   
\end{lemma}

\begin{proof}
  In the Kummer case, $G_d$ consists of the $d$-th roots of unity in
  $\Fqbar$, and we let $G_{d,e}'$ be those of order exactly $e$. Then
  $|G_{d,e}'|$ is independent of $d$, and we set
  $\phi(e)=|G_{d,e}'|$. The orbit of $\Gamma$ through
  $\zeta\in G_{d,e}'$ has size $f$ where $f$ is the smallest positive
  integer such that $\zeta^{q^f}=\zeta$.  Since $\zeta$ has order
  exactly $e$, this is the smallest $f$ such that $q^f\equiv1\pmod e$.
  Clearly this $f$ satisfies $f\ge\log e/\log q$ and this establishes
  Hypothesis~\ref{hyp:big-orbits}.

  In the Artin-Schreier case, $G_d$ is the additive group of
  $\F_{p^d}$, and we let $G_{d,e}'$ consists of those elements of
  $\F_{p^e}\subset\F_{p^d}$ which do not lie in any smaller extension
  of $\Fp$, i.e., $\alpha\in G_{d,e}'$ if and only if
  $\Fp(\alpha)=\F_{p^e}$.  We set $\phi(e)=|G_{d,e}'|$ (which is
  independent of $d$).  Since $\alpha^{p^f}\neq\alpha$ for $0<f<e$, it
  follows immediately that the orbit of the $q$-power Frobenius
  through $\alpha\in G_{d,e}'$ has size at least $e/(\log q/\log p)$,
  and this establishes Hypothesis~\ref{hyp:big-orbits}.

  In the division case, $G_d$ consists of the $\Fqbar$-points of $E$
  of order dividing $d$.  We let $G_{d,e}'$ be the subset of points of
  order exactly $e$, and $\phi(e)=|G_{d,e}'|$ (which is independent of
  $d$).  If $P\in G_{d,e}'$ and $\Fr_q^f(P)=P$, then
  $P\in E(\F_{q^f})$, and this implies that $|E(\F_{q^f})|\ge e$.  But
  the Weil bound implies that $|E(\F_{q^f})|\le(q^{f/2}+1)^2$ which in
  turn implies that $f\ge C\log e$ for some constant $C$ independent
  of $e$.

  In the $\PGL_2$ case, $G_d$ is $\PGL_2(\F_{p^d})$.  For $g\in G_d$,
  let $\Fp(g)$ be defined as follows: choose a representative of $g$
  in $\GL_2(\F_{p^d})$ one of whose entries is 1, and let $\Fp(g)$ be
  the extension of $\Fp$ generated by the other entries.  It is easy
  to see that $\Fp(g)$ is well-defined independent of the choice of
  representative and that $\Fr_p^f(g)=g$ if and only if $\Fr_p^f$
  fixes $\Fp(g)$.  We let $G_{d,e}'$ consists of those elements
  $g\in G_d$ with $\Fp(g)=\F_{p^e}$.  We set $\phi(e)=|G_{d,e}'|$
  (which is independent of $d$).  Since $\Fr_p^f(g)\neq g$ for
  $0<f<e$, it follows immediately that the orbit of the $q$-power
  Frobenius through $g\in G_{d,e}'$ has size at least
  $e/(\log q/\log p)$, and this establishes
  Hypothesis~\ref{hyp:big-orbits}.
\end{proof}

\begin{remss}
  A ``dual'' perspective makes Hypothesis~\ref{hyp:big-orbits} more
  transparent in the cases considered in Lemma~\ref{lemma:H-examples}.
  Namely, let $F=\Fq(\Curve)$ be the function field of a curve of
  genus 0 or 1 over $\Fq$.  (These are the cases where
  $\aut_{\Fqbar}(\Curve)$ is infinite.)  For each $d$, let $G_d$ be a
  subgroup of $\aut_{\Fqbar}(\Curve)$ which is stable under the
  $q$-power Frobenius, and let $\Gamma_d$ be the group of
  automorphisms of $G_d$ generated by Frobenius.  The quotient
  $(\Curve\times\Fqbar)/G_d$ has a canonical model over $\Fq$; let
  $F_d$ be its function field.  With this notation, the extension
  $F/F_d$ is geometrically Galois with group $(G_d,\Gamma_d)$.
  Suppose further that if $e|d$ then $G_e\subset G_d$, so that
  $F_d\subset F_e$.  Then it is natural to define $G_d'$ as the set of
  elements in $G_d$ which are not in $G_e$ for any divisor of $d$ with
  $e<d$.  Clearly $G_e'$ depends only on $e$, and the decompostion
  $G_d=\cup_{e|d}G_e'$ is evident.  All of the examples of
  Lemma~\ref{lemma:H-examples} can be recast in this form.
\end{remss}

The following lemma is modeled on \cite[Lemme~3.1.1]{GriffonThesis}.

\begin{lemma}\label{lemma:orbit-bound}
  Let $K_d/K$ be a tower of geometrically Galois extensions such that
  for all $d$, $|G_d|\ge d$ and such that
  Hypothesis~\ref{hyp:big-orbits} holds.  Then there is a constant
  $C_1$ such that the number of orbits of $\Gamma_d$ on $G_d$
  satisfies
$$\left|G_d/\Gamma_d\right|  \le C_1\frac{|G_d|}{\log|G_d|}$$
for all $d>1$.
\end{lemma}

\begin{proof}
  Let $\psi(d)=|G_d|$, so that $\psi(d)=\sum_{e|d}\phi(e)$.  Extend
  $\psi$ to a function of real numbers which is continuous,
  increasing, and satisfies $\psi(x)\ge x$ for all $x$. By
  Hypothesis~\ref{hyp:big-orbits}, for all $d>1$ the number of orbits
  of $\Gamma_d$ on $G_{d,e}'$ satisfies
$$\left|G_{d,e}'/\Gamma_d\right|  \le C^{-1}\frac{\phi(e)}{\log\psi(e)}.$$
Let $x>1$ be a parameter to be chosen later.  We have
\begin{align*}
\left|G_d/\Gamma_d\right|  
&\le C_2\sum_{1<e|d}\frac{\phi(e)}{\log\psi(e)}
&\text{($C_2$ to compensate for omitting $e=1$)}\\
&= C_2\sum_{\substack{1<e|d\\ e\le x}}\frac{\phi(e)}{\log\psi(e)}
   +C_2\sum_{\substack{{1<e|d}\\{e> x}}}\frac{\phi(e)}{\log\psi(e)}\\
&\le C_2\sum_{\substack{{1<e|d}\\{e\le x}}}\frac{\phi(e)}{\log\psi(e)}
   +C_2\frac{\psi(d)}{\log\psi(x)}
&\text{($\sum\phi(e)=\psi(d)$ and $\psi$ increasing)}\\
&\le C_2\sum_{\substack{{1<e|d}\\{e\le x}}}\frac{\psi(e)}{\log\psi(e)}
   +C_2\frac{\psi(d)}{\log\psi(x)}
&\text{($\phi(e)\le\psi(e)$)}\\
&\le C_3\frac{\psi(x)}{\log\psi(x)}\sum_{\substack{{1<e|d}\\{e\le x}}}1
   +C_2\frac{\psi(d)}{\log\psi(x)}
&\text{\stackunder{($x\mapsto\psi(x)\mapsto\psi(x)/\log\psi(x)$}
{$\qquad$ increasing for $x>2.72$)}}\\
\displaybreak
&\le C_3\frac{x\psi(x)}{\log\psi(x)}
   +C_2\frac{\psi(d)}{\log\psi(x)}\\
&\le C_3\frac{\psi(x)^2}{\log\psi(x)}
   +C_2\frac{\psi(d)}{\log\psi(x)}
&\text{($\psi(x)\ge x$)}
\end{align*}
Now since $\psi$ is increasing and continuous, we may choose $x$ so that
$\psi(x)^2=\psi(d)$, and for this choice we have
$$\left|G_d/\Gamma_d\right|  \le
\left(2C_3+2C_2\right)\frac{\psi(d)}{\log\psi(d)}.$$
Thus setting $C_1=2C_3+2C_2$ completes the proof.
\end{proof}

We now consider the set of orbits of $\Gamma$ on a homogeneous space
for $G$.

\begin{lemma}\label{lemma:orbit-inequality}
  Let $G$ be a finite group and let $T$ be a principal homogeneous
  space for $G$.  Let $\Gamma$ be a group acting on $G$ \textup{(}by
  group automorphisms\textup{)} and on $T$ \textup{(}by
  permutations\textup{)}, and suppose that the actions of $\Gamma$ on
  $G$ and $T$ are compatible with the action of $G$ on $T$ (i.e., for
  all $\gamma\in\Gamma$, $g\in G$, and $t\in T$,
  $\gamma(gt)=\gamma(g)\gamma(t)$. Then
$$\left|T/\Gamma\right|\le\left|G/\Gamma\right|.$$
\end{lemma}

\begin{proof}
We use the orbit counting lemma:
$$\left|G/\Gamma\right|=\frac1{|\Gamma|}
\sum_{\gamma\in\Gamma}\left|G^\gamma\right|$$
where $G^\gamma$ denotes the set of fixed points of $\gamma$ acting on
$G$.  Similarly,
$$\left|T/\Gamma\right|=\frac1{|\Gamma|}
\sum_{\gamma\in\Gamma}\left|T^\gamma\right|$$
where $T^\gamma$ denotes the set of fixed points of $\gamma$ acting on
$G$.  We claim that if $T^\gamma$ is not empty, then it is a principal
homogeneous space for $G^\gamma$.  Indeed, it is clear that if
$g\in G^\gamma$ and $t\in T^\gamma$, then $gt\in T^\gamma$.
Conversely, if $t,t'\in T^\gamma$ and $g\in G$ is the unique element such
that $gt=t'$, then 
$$\gamma(g)t=\gamma(g)\gamma(t)=\gamma(gt)=\gamma(t')=t'=gt,$$
and so $\gamma(g)=g$.  Therefore, for each $\gamma\in\Gamma$,
$\left|T^\gamma\right|\le\left|G^\gamma\right|$.   We conclude that
$$|T/\Gamma|=\frac1{|\Gamma|}\sum_{\gamma\in \Gamma}|T^\gamma|\le
\frac1{|\Gamma|}\sum_{\gamma\in \Gamma}|G^\gamma|=|G/\Gamma|,$$
and this completes the proof of the lemma.
\end{proof}

\begin{remss}\label{rem:orbits}
In fact, the conclusion of the lemma holds when we assume only that
$G$ acts transitively on $T$.  To see this, it suffices to check that
for all $\gamma\in\Gamma$, $|T^\gamma|\le|G^\gamma|$.  If $T^\gamma$
is empty, there is nothing to prove.  If not, choose $t_0\in
T^\gamma$, let $G_0$ be the stabilizer of $t_0$ in $G$, 
and set
\begin{align*}
F(\gamma)&=\left\{g\in G\left|\gamma(gt_0)=gt_0\right.\right\}\\
&=\left\{g\in G\left|g^{-1}\gamma(g)\in G_0\right.\right\}.
\end{align*}
Then $G_0$ acts freely on $F(\gamma)$ by right multiplication, and the
quotient is $T^\gamma$.  Thus 
$|F(\gamma)|=|G_0|\cdot|T^\gamma|$.
On the other hand, $G^\gamma$ acts freely on $F(\gamma)$ by left
multiplication, and the quotient maps injectively to $G_0$ by $g\mapsto
g^{-1}\gamma(g)$.  Thus we find
$$|G_0|\cdot|T^\gamma|=|F(\gamma)|\le|G^\gamma|\cdot|G_0|$$
and so $|T^\gamma|\le|G^\gamma|$.  It is also clear that $G^\gamma
G_0\subset F(\gamma)$ so in all we have
$$\frac{|G^\gamma|}{|G_0^\gamma|}\le |T^\gamma|\le|G^\gamma|.$$
Simple examples show that both bounds are sharp.  Thanks to Alex Ryba
for the proofs in this remark and the preceding lemma.
\end{remss}

% The lower bound is sharp when $G^\gamma$ acts transitively on $T^\gamma$.
% For an example in the middle, consider $G=S_3$ acting on
% $T=\{(123),(132)\}$ by conjugation.  Let $\gamma$ be conjugation by
% $(123)$ on both $G$ and $T$.  This is compatible with action.  If
% $t_0=(123)$, then $G^\gamma=G_0=\{id,(123),(132)\}$ whereas
% $T^\gamma=T$, so 
% $$1=\frac{|G^\gamma|}{|G_0^\gamma|}\le 2=|T^\gamma|\le3=|G^\gamma|.$$
% The upper bound is sharp when $G_0$ is trivial.

\begin{corss}\label{cor:place-bounds}
Suppose that $K_d$ is a tower of geometrically Galois extensions of $K$
such that $[K_d:K]\ge d$ and such that 
Hypothesis~\ref{hyp:big-orbits} holds.  Let $v$ be a place of $K$.
%which is unramified in each $K_d$.  
Then there is a constant $C_v$ depending only on $K$ and $v$ such that
for all $d$ the number of place of $K_d$ over $v$ is at most
$C_v[K_d:K]/\log[K_d:K]$.
\end{corss}

\begin{proof}
  First assume that $v$ is unramified in $K_d$.  Let $T_d$ be the set
  of geometric points in the fiber over $v$ (i.e., in the fiber of the
  map of curves corresponding to the extension $K_d/K$) and let
  $G=G_d$ be the geometric Galois group of $K_d$ over $K$.  Let $k_v$
  be the residue field at $v$ and let $\Gamma_d=\gal(k_d/k_v)$, a
  subgroup of the Galois group of the splitting field of $G_d$.  Then
  $T_d$ is a principal homogeneous space for $G_d$, and $\Gamma_d$
  acts on $G_d$ and $T_d$ compatibly with the action of $G_d$ on
  $T_d$.  By Lemma~\ref{lemma:orbit-inequality},
  $\left|T_d/\Gamma_d\right|\le\left|G_d/\Gamma_d\right|$.  But
  $T_d/\Gamma_d$ is in bijection with the set of places of $K_d$ over
  $v$, and by Lemma~\ref{lemma:orbit-bound} (applied to the extensions
  $k_vK_d/k_vK$), there is a constant $C_v$ (depending on $v$ because
  the tower in question depends on $v$) such that
$$\left|G_d/\Gamma_d\right|\le C_v\frac{[K_d:K]}{\log[K_d:K]}.$$
This completes the proof of the corollary when $v$ is unramified in
$K_d$.  The general case follows from the same argument using
Remark~\ref{rem:orbits} in place of Lemma~\ref{lemma:orbit-inequality}.
\end{proof}

\subsection{Bounds on Tamagawa numbers (2)}\label{ss:T2}
We now turn to a second improvement on the Hindry-Pacheco bound on
Tamagawa numbers.  We consider towers of fields satisfying the
conclusions of Lemma~\ref{lemma:easy-place-bounds} and
Corollary~\ref{cor:place-bounds}, and we bound Tamagawa numbers using
only a mild (local) semi-stability hypothesis and no restriction on
the characteristic of the ground field.

Recall the line bundle $\omega_{A}$ associated to an abelian variety
$A$ defined in Section~\ref{ss:prelims}.

\begin{propss}\label{prop:general-tau-bound}
  Let $K$ be a global function field of characteristic $p$, let $Z$ be
  a finite set of places of $K$, and let $K_d$ be a tower of
  geometrically Galois extensions of $K$.  Assume that $[K_d:K]\ge d$
  and that for each place $v$ of $K$ there is a constant $C_v$ such
  that the number of places of $K_d$ dividing $v$ is
  $\le C_v[K_d:K]/\log[K_d:K]$ for all $d$.  Suppose that each $K_d/K$
  is unramified outside $Z$.  Let $A$ be an abelian variety over $K$
  which has semi-stable reduction at each place $v\in Z$ and such that
  $\deg\omega_A>0$.  Let $A_d=A\times_K K_d$.  Then
$$\tau(A_d)=O(H(A_d)^\epsilon)$$
for every $\epsilon>0$.
\end{propss}

\begin{proof}
  To lighten notation, let $D=[K_d:K]$.  Since $A$ has semi-stable
  reduction at the possibly ramified places $Z$, we have
  $\deg\omega_{A_d}=D\deg\omega_A\ge D$, so it will suffice to
  show that
$$\tau(A_d)=O\left(q^{D\epsilon}\right)$$ 
for all $\epsilon>0$.

For each place $v$ of $K$, let $c_v$ be the order of the group of
connected components of the special fiber of the N\'eron model of $A$
at $v$.  Let $\overline{c}_v$ be the order of the group of connected
components of the special fiber of the N\'eron model of $A$ at a place
of $\Fqbar K$ over $v$.  (The order is independent of the choice.)
Since the former group is a subgroup of the latter, $c_v$ divides
$\overline{c}_v$.  If $w$ is a place of $K_d$ over $v$, let $c_w$ be
the order of the component group of the N\'eron model of $A$ over
$K_d$.

Consider a place $v\not\in Z$.  Since $K_d/K$ is unramified at $v$,
$c_w$ divides $\overline{c}_v$.  By assumption,
the number of places $w$ over $v$ is bounded by $C_vD/\log D$.  Since
there are only finitely many places of $K$ where $A$ has bad
reduction, we may set
$C_1=\max\{\overline{c}_v^{C_v}|v\text{ of bad reduction}\}$ and
conclude that
$$\prod_{w|v\not\in Z}c_w\le\prod_{v\not\in
  Z}\overline{c}_v^{C_vD/\log D}\le C_1^{D/\log D}.$$

Now consider a place $v\in Z$, let $w$ be a place of $K_d$ over $v$,
and let $r$ be the ramification index of $w$ over $v$.  Since $K_d/K$
is geometrically Galois, $r$ depends only on $v$.  Since $A$ is
assumed to have semi-stable reduction, \cite[Thm~5.7]{HalleNicaise10}
implies that 
$$c_w\le \overline{c}_v r^{\dim(A)}.$$
Moreover, by assumption, the number of places of
$K_d$ over $v$ is at most $\min\{D/r,C_vD/\log D\}$ for some constant
$C_v$ which is independent of $D$.  If $r\le(\log D)/C_v$, we have
$$\prod_{w|v}c_w
\le\left(\overline{c}_v r^{\dim(A)}\right)^{C_vD/\log  D}
\le C_2^{D/(\log D/\log\log D)}$$
where $C_2$ depends only on $v$ and $A$.
If $r\ge(\log D)/C_v$, we have
$$\prod_{w|v}c_w\le\left(\overline{c}_v r^{\dim(A)}\right)^{D/r}
\le C_3^{D\log r/r}\le C_4^{D/(\log D/\log\log D)}$$
where again $C_3$ and $C_4$ depend only on $v$ and $A$.

Taking the product over all place $w$ of $K_d$ and setting
$C_5=\max\{C_2,C_4\}$, we have
$$\prod_{w}c_w
=\left(\prod_{w|v\not\in Z}c_w\right)\left(\prod_{w|v\in Z}c_w\right)
\le \left(C_1^{D/\log D}\right) \left(C_5^{D/(\log D/\log\log D)}\right)^{|Z|}$$
and this is clearly $O(q^{D\epsilon})$ as $d$ (and therefore $D$) tends
to infinity.
\end{proof}

We now give the main application of the results in this section.
Assume $K=\Fq(t)$ or $K=\Fq(E)$ for an elliptic curve $E$, and
consider a family of abelian varieties $A_d$ over $K$ associated to
the Kummer, Artin-Schreier, division, or $\PGL_2$ towers.  Recall the line
bundle $\omega=\omega_{A}$ defined in the Section~\ref{ss:prelims}.

\begin{corss}\label{cor:tau-av}
  As $d$ runs through positive integers prime to $p$ \textup{(}or all
  positive integers in the Artin-Schreier case\textup{)}, we have
$$\tau(A_d)=O(H(A_d)^\epsilon)$$
for every $\epsilon>0$ in any of the following situations:
\begin{enumerate}
\item $A$ is an abelian variety over $K=\Fq(t)$,  $A_d$ is
  the family associated to the Kummer tower,
  $\deg(\omega)>0$, and $A$ has semi-stable reduction at
  $t=0$ and $t=\infty$.
\item $A$ is an abelian variety over $K=\Fq(t)$, $A_d$ is the family
  associated to the Artin-Schreier tower, $\deg(\omega)>0$, and
  $A$ has semi-stable reduction at $t=\infty$.
\item $A$ is an abelian variety over $K=\Fq(E)$,  $A_d$ is
  the family associated to the division tower, and
  $\deg(\omega)>0$.
\item $A$ is an abelian variety over $K=\Fq(t)$, $A_d$ is the family
  associated to the $\PGL_2$ tower, $\deg(\omega)>0$, and
  $A$ has semi-stable reduction at $t=0$ and $t=1$.
\end{enumerate}
\end{corss}

\begin{proof}
  This is an immediate consequence of
  Proposition~\ref{prop:general-tau-bound} together with
  Lemma~\ref{lemma:easy-place-bounds}.
\end{proof}

\subsection{Bounds on Tamagawa numbers (3)}\label{ss:T3}
\label{ss:tau-tame}
Our third improvement on the Hindry-Pacheco bound on Tamagawa numbers
is to note that we can get by with a weaker hypotheses in case (1) of
Corollary~\ref{cor:tau-av}.  Namely, we claim that the conclusion of
the corollary holds if there exists an integer $e$ relatively prime to
$p$ such that $A$ has semi-stable reduction at the places $u=0$ and
$u=\infty$ of $\Fq(u)$ where $u^e=t$.  (The corollary is the case
where $e=1$.)

  To check the claim, we first recall a result of Halle and Nicaise:
  Let $A$ be an abelian variety over $\Fpbar((t))$.  For $d$ prime
  to $p$, let $c_d$ denote the order of the component group of the
  special fiber of the N\'eron model of $A$ over $\Fpbar((t^{1/d}))$.
  Then \cite[Thm.~6.5]{HalleNicaise10} says that if we assume that $A$
  acquires semi-stable reduction over $\Fpbar((t^{1/e}))$ for some $e$
  prime to $p$, then the series
$$\sum_{(p,d)=1}c_dT^d$$
is a rational function in $T$ and $1/(T^j-1)$ for $j\ge1$.  This implies in
particular that the $c_d$ have at worst polynomial growth:
$c_d=O(d^N)$ for some $N$.

Applying this result in the context of part (1) of the lemma for the
places $t=0$ and $t=\infty$ of $\Fq(t)$, we see that
$$\tau(A_d)\le C_1^{d/\log d}d^{C_6}=O(H(A_d)^\epsilon)$$
for all $\epsilon>0$.

\subsection{Estimating $\deg(\omega_J)$}
When $A=J$ is the Jacobian of a curve $X$ over a function field,
computing $\deg(\omega_J)$ typically involves knowledge of a regular
model of $X$ (or a mildly singular model), information which is
sometimes difficult to obtain.  The following lemma allows us to
reduce to easy cases in two examples later in the paper.

\begin{lemma}\label{lemma:deg-omega}
  Let $K=k(\CC)$ be the function field of a curve over a perfect field
  $k$.  Let $X$ be a smooth, projective curve of genus $g$ over $K$.
  Let $J$ be the Jacobian of $X$, let $\pi:\XX\to\CC$ be a regular minimal
  model of $X$ over $K$, and let $\JJ\to\CC$ be the N\'eron model of
  $J$ with zero-section $z:\CC\to\JJ$.  Let 
$$\omega_J:=\bigwedge^g\left(z^*\Omega^1_{\JJ/\CC}\right)$$
be the Hodge bundle of $J$.

Let $K'$ be a finite, separable, geometric extension of $K$, and let
$\rho:\CC'\to\CC$ be the corresponding morphism of curves over $k$.
Let $R=(2g_{\CC'}-2)-[K':K](2g_\CC-2)$.  

Let $X'=X\times_KK'$ with Jacobian $J'$, models $\XX'$ and $\JJ'$, and
Hodge bundle $\omega_{J'}$.  Then
$$[K':K]\deg(\omega_{J})\le \deg(\omega_{J'})+gR.$$
\end{lemma}

The point of the lemma is that we do not lose much information in
passing to a finite extension.

\begin{proof}[Proof of Lemma~\ref{lemma:deg-omega}]
    Since $\XX$ is regular and $\pi$ has a section, we have that
    \begin{align*}
\omega_J&\cong\bigwedge^g\left(\pi_*\Omega^2_{\XX/k}
\tensor\left(\Omega^1_{\CC/k}\right)^{-1}\right)\\ 
&\cong\left(\bigwedge^g\pi_*\Omega^2_{\XX/k}\right)
\tensor\left(\Omega^1_{\CC/k}\right)^{\tensor^{-g}} 
    \end{align*}
    and similarly for $\omega_{J'}$.  This argument, which uses
    results on N\'eron models and relative duality, is given in the
    proof of \cite[Prop.~7.4]{AIMgroup}.

There is a dominant rational map $\XX'\ratto\XX$ covering $\rho$, so
pull back of 2-forms induces a non-zero morphism of sheaves
$$\rho^*\bigwedge^g\left(\pi_*\Omega^2_{\XX/k}\right)
\to\bigwedge^g\left(\pi'_*\Omega^2_{\XX'/k}\right).$$

By Riemann-Hurwitz, we have 
$$\rho^*\left(\Omega^1_{\CC/K}\right)\cong\Omega^1_{\CC'/k}\tensor\OO_{\CC'}(D)$$
where $D$ is a divisor on $\CC'$ of degree $R$.

Thus we get a non-zero morphism of sheaves
$$\rho^*(\omega_J)\to\omega_{J'}\tensor\OO_{\CC'}(gD).$$

Taking degrees, we conclude that
$$[K':K]\deg(\omega_J)\le\deg(\omega_{J'})+gR$$
as desired.
\end{proof}

\section{Integrality of the regulator and general lower bounds}
In this section, we give a lower bound on the regulator $\Reg(A)$ in
terms of Tamagawa numbers.  Combined with the bounds on $\tau(A)$
given in the preceding section, this yields a lower bound on the
Brauer-Siegel ratio.  A more general version of the same lower bound was
proven in \cite[Prop.~7.6]{HindryPacheco16}, but our proof is arguably
simpler and more uniform, and avoids a forward reference in
\cite{HindryPacheco16}.

\subsection{Integrality of regulators}
We continue with the standard notations introduced in
Section~\ref{s:prelims}.  In particular, $A$ is an abelian variety over
the function field $K=k(\Curve)$ with N\'eron model $\AA$ and dual
abelian variety $\hat A$.  We consider the height pairing
$A(K)\times\hat A(K)\to\Q$ (which we recall is the canonical
N\'eron-Tate height divided by $\log q$ and which takes values in
$\Q$) and its determinant $\Reg(A)$.

Our main goal in this section is to bound the denominator of the
regulator in terms of the orders $c_v$ of the component groups of
$\AA$ at places $v$ of $K$.  Recall that $\tau(A)=\prod_vc_v$.

\begin{propss}\label{prop:integrality}
The rational number
$$\tau(A)\Reg(A)$$
is an integer.
\end{propss}

\begin{proof}
  We refer to \cite{HindrySilvermanDG} for general background on
  heights.  Given an invertible sheaf $\LL$ on $A$ and a point
  $x\in A(K)$, the general theory of heights on abelian varieties
  defines a rational number $h_\LL(x)$.  The canonical height pairing
  we are discussing is defined using this machine and the
  identification of $\hat A$ with $\Pic^0(A)$, the group of invertible
  sheaves algebraically equivalent to zero.  In other words, given
  $x\in A(K)$ and $y\in \hat A(K)$, we take $\LL$ to be the invertible
  sheaf associated to $y$ and define
$$\<x,y\>=h_\LL(x).$$

N\'eron's theory \cite{Neron65} decomposes the height $h_\LL(x)$ into
a sum of local terms indexed by the places of $K$.  In
\cite[III.1]{Moret-Bailly85}, Moret-Bailly proves that the contribution at a
place $v$ has denominator at most $2c_v$, and at most $c_v$ if $c_v$
is odd.  Moreover, he gives an example which shows that this is in
general best possible.  The upper bound on the denominator comes from
a property of ``pointed maps of degree 2,''
\cite[I.5.6]{Moret-Bailly85}, namely that a pointed map of degree 2
from a group of exponent $n$ has exponent at worst $2n$, or $n$ if $n$
is odd.  (These terms will be defined just below.)

In our situation there is slightly more structure: Since $\LL$ is
algebraically equivalent to zero, it is anti-symmetric, i.e., if
$[-1]$ is the inverse map on $A$, the $[-1]^*\LL\cong\LL^{-1}$.  The
functoriality in \cite[III.1.1]{Moret-Bailly85} then shows that the
corresponding pointed map of degree 2 is also anti-symmetric.  In the
next lemma, we define anti-symmetric pointed maps of degree 2, and we
prove that such a map from a group of exponent $c$ has exponent
dividing $c$.

Thus we see that $\<x,y\>$ is a sum of local terms, and the term at
a place $v$ has denominator at worst $c_v$.  It follows from the
bilinearity of the local terms $\<,\>_v$ that if $x$ passes through the identity
component at $v$, then $\<x,y\>_v$ is an integer.  We
define a ``reduced Mordell-Weil group''
$$A(K)^{red}:=\left\{\left.x\in A(K)\right|\text{$x$ meets the
    identity component of $\AA$ at every $v$}\right\},$$ 
and note that if $x\in A(K)^{red}$, then $\<x,y\>$
is an integer for every $y\in\hat A(K)$.  Since the index of
$A(K)^{red}$ in $A(K)$ divides $\tau(A)=\prod_vc_v$, we see that 
$$\Reg(A)\in\tau^{-1}\Z$$ 
as desired.  The proposition thus follows from the next lemma.
\end{proof}

\begin{lemma}
  Let $A$ and $G$ be abelian groups and let $f:A\to G$ be a function
  such that:
  \begin{enumerate}
  \item $f$ is a ``pointed map of degree 2,'' namely,
$$f(x_1+x_2+x_3)-f(x_1+x_2)-f(x_1+x_3)-f(x_2+x_3)
+f(x_1)+f(x_2)+f(x_3)=0$$
for all $x_1,x_2,x_3\in A$;
\item and $f$ is ``anti-symmetric,'' i.e., $f(-x)=-f(x)$ for all $x\in
  A$.
\end{enumerate}
Then for all integers $n$ and all $x\in A$, $f(nx)=nf(x)$.  In
particular, if $A$ has exponent $c$, then $cf=0$, i.e., $cf(x)=0$ for
all $x\in A$.
\end{lemma}

\begin{proof}
  This follows from a simple inductive argument.  Clearly it suffices
  to treat the case $n\ge0$.  Taking $x_1=x_2=x_3=0$ in the pointed
  map property shows that $f(0)=0$.  Taking $x_1=x_2=x$ and $x_3=-x$
  then shows that $f(2x)=2f(x)$.  Finally, for $n\ge2$, taking $x_1=(n-1)x$,
  $x_2=x_3=x$, we have 
  \begin{align*}
   f((n+1)x)&= f((n-1)x+x+x)\\
&=f(nx)+f(nx)+f(2x)-f((n-1)x)-f(x)-f(x)\\
&=\left(n+n+2-(n-1)-1-1\right)f(x)\\
&=(n+1)x
  \end{align*}
  where we use induction to pass from the second displayed line to the
  third.  This yields the lemma.
\end{proof}

Without the anti-symmetry hypothesis, we would have
$$f(nx)=\frac{n(n+1)}2f(x)+\frac{n(n-1)}2f(-x),$$
by the same argument leading from the theorem of the cube
\cite[A.7.2.1]{HindrySilvermanDG} to Mumford's formula
\cite[A.7.2.5]{HindrySilvermanDG}.

\subsection{Further comments on integrality}
Let $\XX\to\CC$ be a fibered surface with generic fiber $X/K$ and
assume $X$ has a $K$-rational point.  Let $A$ be the Jacobian $J_X$.
In \cite[Prop.~7.2]{AIMgroup}, we proved that the rational number
\begin{equation}\label{eq:AIMintegrality}
\frac{|\NS(\XX)_{tor}|^2}{|A(K)_{tor}|^2}%\left(\prod_v c_v\right)
\tau(A)\Reg(A)
\end{equation}
is an integer.  (By the factorization of birational maps into blow-ups
and the blow-up formula, $\NS(\XX)_{tor}$ is a birational invariant,
so the displayed quantity depends only on $X$ and $K$.)

Note that this bound on the denominator of $\Reg(A)$ is in general
stronger than that of Proposition~\ref{prop:integrality}.  For
example,  for the Jacobians studied in \cite{Ulmer14a} and \cite{AIMgroup},
\eqref{eq:AIMintegrality} is stronger than Proposition~\ref{prop:integrality}.

When $X$ has genus 1, it is known that
$\NS(\XX)_{tor}$ is trivial, so \eqref{eq:AIMintegrality} says that
\begin{equation}\label{eq:bad-integrality-A}
\frac{\tau(A)}{|A(K)_{tor}|^2}\Reg(A)\in\Z
\end{equation}
This bound (unlike \eqref{eq:AIMintegrality}) makes sense for general
abelian varieties, and it is reasonable to ask whether it holds in
general.  In the rest of this subsection, we sketch a proof that
\eqref{eq:bad-integrality-A} does not hold in general, not even for
Jacobians over $\Fq(t)$.

Let $\YY$ be a classical Enriques surface over $\Fp$.  It is known
that 
$$\NS(\YY)_{tor}\cong\Z/2\Z,\qquad\NS(\YY)/tor\cong\Z^{10},
\quad\text{and}\quad\det(\NS(\YY))=1;$$ 
see \cite{CossecDolgachevES}.

Next, embed $\YY$ in some projective space and take a Lefschetz
pencil, extending $\Fp$ to $\Fq$ if necessary.  Let $\XX$ be the
result of blowing up $\YY$ at the base points of the pencil.  Thus we
have $\pi:\XX\to\P^1$ over $\Fq$ whose fibers are
irreducible and either smooth or with single a node.  Moreover $\pi$ has
a section.  Choose such a section $O$ and a fiber $F$.  We have
intersection pairings $O^2=-1$, $F^2=0$, and $F.O=1$.  Also, the
N\'eron-Severi groups satisfy
$$\NS(\XX)=\NS(\YY)\oplus\<-1\>^d$$
where the direct sum is orthogonal, $\<-1\>$ stands for a
copy of $\Z$ whose generator has self-intersection $-1$, and $d$ is
the number of blow-ups.  Thus $\det(\NS(\XX))=1$.

Let $X/K=\Fq(t)$ be the generic fiber of $\pi$.  This is a smooth
curve with a $K$-rational point.  Let $A=J_X$ be its Jacobian.  We
will see shortly that $A$ is a counterexample to
\eqref{eq:bad-integrality-A}.

Since $\Pic^0(\XX)=\Pic^0(\YY)=0$, we have $\tr_{K/\Fq}(A)=0$.  The
Shioda-Tate theorem gives an exact sequence
$$0\to(\Z O+\Z F)\to\NS(\XX)\to A(K)\to 0.$$
Moreover, the fact that $\pi$ has irreducible fibers implies that
there is a splitting $A(K)\to\NS(\XX)$ which sends the canonical
height (divided by $\log q$) to the intersection pairing on
$\NS(\XX)$.  It follows from the intersection formulas for $O$ and
$F$ noted above that 
$$\Reg(A):=\det(A(K)/tor)=\det(NS(\XX)/tor)=1.$$

Since $\pi$ has irreducible fibers, $\tau(A)=1$.  The Shioda-Tate exact
sequence above shows that $A(K)_{tor}$ has order at least 2 (in fact,
exactly 2), so 
$$\frac{\tau(A)}{|A(K)_{tor}|^2}\Reg(A)=\frac14.$$
Thus \eqref{eq:bad-integrality-A}. fails for $A$.

\subsection{Lower bounds on Brauer-Siegel ratio from integrality}

We now state the main consequence for the Brauer-Siegel ratio of our
Proposition~\ref{prop:integrality}.

\begin{propss}\label{prop:BS=>0}
  Let $A_d$ be a family of abelian varieties over $K$ with
  $H(A_d)\to\infty$.  Assume that $\tau(A_d)=O(H(A_d)^\epsilon)$ for
  all $\epsilon>0$.  Then $\liminf \BS(A_d)\ge0$.
\end{propss}

\begin{proof}
  Noting that $|\sha(A_d)|$ is a positive integer and is therefore
  $\ge1$, we have that
$$\log\left(|\sha(A_d)|\Reg(A_d)\right)\ge\log\left(\Reg(A_d)\right).$$
Proposition~\ref{prop:integrality}
implies that
$$\log\left(\Reg(A_d)\right)\ge-\log\left(\tau(A_d)\right).$$
It follows from the hypothesis $\tau(A_d)=O(H(A_d)^\epsilon)$ that
$$\BS(A_d)=\frac{\log\left(|\sha(A_d)|\Reg(A_d)\right)}{\log\left(H(A_d)\right)}
\ge\frac{-\log\left(\tau(A_d)\right)}{\log\left(H(A_d)\right)}$$
has $\liminf\ge0$ as $d\to\infty$.
\end{proof}

\begin{corss}
  If $A_d$ is a family of abelian varieties over $K$ such that
  $H(A_d)\to\infty$, then $\liminf \BS(A_d)\ge0$ in any of the
  following situations:
\begin{enumerate}
\item $\dim(A_d)=1$ for all $d$
\item $A$ is an abelian variety over $K=\Fq(t)$,  $A_d$ is
  the family associated to $A$ and the Kummer tower,
  and $A$ has semi-stable reduction at
  $t=0$ and $t=\infty$.
\item $A$ is an abelian variety over $K=\Fq(t)$,  $A_d$ is
  the family associated to $A$ and the Artin-Schreier tower,
  and $A$ has semi-stable reduction at
  $t=\infty$.
\item $A$ is an abelian variety over $K=\Fq(E)$,  and $A_d$ is
  the family associated to $A$ and the division tower.
\item $A$ is an abelian variety over $K=\Fq(t)$,  $A_d$ is
  the family associated to $A$ and the $\PGL_2$ tower,
  and $A$ has semi-stable reduction at
  $t=0$ and $t=1$.
\end{enumerate}
\end{corss}

\begin{proof}
  This is immediate from Lemma~\ref{lemma:tau-ec},
  Corollary~\ref{cor:tau-av}, and Proposition~\ref{prop:BS=>0}.
\end{proof}

\section{Lower bounds via the dimension of the Tate-Shafarevich
  functor}\label{s:dim-sha} 
In this section, we assume that the conjecture of Birch and
Swinnerton-Dyer (more precisely, the finiteness of $\sha(A)$) holds
for all abelian varieties considered.  Given an abelian variety $A$
over $K=\Fq(\CC)$, we will consider the functor from finite extensions
of $\Fq$ to groups given by
$$\Fqn\mapsto \sha(A\times_{\Fq(\CC)} \Fqn(\CC))$$
and we will show that the dimension of this functor (to be defined
below) gives information on the Brauer-Siegel ratio of $A$ over $K$.
This technical device will be extremely convenient as it allows us to
bound the Brauer-Siegel ratio without considering the regulator.

\begin{prop-def}\label{prop:sha-dim}
For each positive integer $n$, let $K_n:=\Fqn(\CC)$.  Given an abelian
variety $A$ over $K=K_1$, write $A/K_n$ for $A\times_KK_n$.  
Then the limit 
$$\lim_{n\to\infty}\frac{\log|\sha(A/K_n)[p^\infty]|}{\log(q^n)}$$
exists and is an integer.  We call it the \emph{dimension of $\sha(A)$}, and
denote it $\dim\sha(A)$.
\end{prop-def}

The proof of the proposition will be given later in this section,
after giving a formula for $\dim\sha(A)$ in terms of the $L$-function
of $A$.  We give a justification of the terminology ``dimension'' in
Remark~\ref{rems:dim} below.

In order to state a formula for $\dim\sha(A)$, we recall some
well-known results on the $L$-function $L(A,s)$.  Let
$A_0=\tr_{K/k}(A)$ be the $K/k$-trace of $A$ (where as usual $k=\Fq$),
an abelian variety over $k$.  %, and let $d_0=\dim A_0$.
(See \cite{Conrad06} for a modern account of the $K/k$-trace.)  Then
$L(A,s)$ has the form
$$L(A,s)=\frac{P(q^{-s})}{Q(q^{-s})Q(q^{1-s})}$$
where $P$ and $Q$ are polynomials with the following properties:
\begin{enumerate}
\item $P(T)=\prod_i(1-\alpha_iT)$ where the $\alpha_i$ are Weil
  numbers of size $q$.
\item $Q$ has degree $2\dim(A_0)$ and $Q(T)=\prod_j(1-\beta_jT)$ where the
  $\beta_j$ are the Weil numbers of size $q^{1/2}$ associated to
  $A_0$.  (In other words, they are the eigenvalues of 
  Frobenius on $H^1(A_0\times\Fqbar,\Ql)$ for any $\ell\neq p$.)
\item $Q(1)=|A_0(\Fq)|$ and $Q(q^{-1})=q^{-d_0}|A_0(\Fq)|$.
\item Replacing $A$ with $A/K_n$ has the effect of replacing the
  $\alpha_i$ and $\beta_j$ with $\alpha_i^n$ and $\beta_j^n$. 
\end{enumerate}

Let $F$ be the number field generated by the $\alpha_i$, and choose a
prime of $F$ over $p$ with associated valuation $v$ normalized so that
$v(q)=1$.  We define the \emph{slopes} associated to $A$ to be the
rational numbers $\lambda_i=v(\alpha_i)$.  It is known that the set of
slopes (with multiplicities) is independent of the choice of $v$, that
$0\le\lambda_i\le2$ for all $i$, and that the set of slopes is
invariant under $\lambda_i\mapsto2-\lambda_i$.

We can now state a formula for the dimension of $\sha(A)$.  

\begin{prop}\label{prop:sha-formula}
  $$\dim\sha(A)=
\deg(\omega)+\dim(A)(g_\CC-1)+\dim(A_0)
  -\sum_{\lambda_i<1}\left(1-\lambda_i\right).$$
The last sum is over indices $i$ such that $\lambda_i<1$.
\end{prop}

Before giving the proof of Propositions~\ref{prop:sha-dim} and
\ref{prop:sha-formula}, we record an elementary lemma on $p$-adic
numbers. 

\begin{lemma}\label{lemma:ords}
  Let $E$ be a finite extension of
  $\Qp$, let $\m$ be the maximal ideal of $E$, and let
  $\ord:E^\times\to\Z$ be the valuation of $E$.  If
  $\gamma\in E^\times$ has $\ord(\gamma)=0$ and is not a root of
  unity, then
$$\ord\left(1-\gamma^n\right)=O(\log n).$$
\end{lemma}

\begin{proof}
  First we note that replacing $\gamma$ with $\gamma^a$, we may assume
  without loss of generality that $\gamma$ is a 1-unit, i.e., that
  $\ord(1-\gamma)>0$.  Next, if $n=p^em$ with $p\nodiv m$, then
$$\frac{1-\gamma^n}{1-\gamma^{p^e}}=
1+\gamma^{p^e}+\cdots+\gamma^{p^e(m-1)}
\equiv m \not\equiv0\pmod\m,$$
so $\ord(1-\gamma^n)=\ord(1-\gamma^{p^e})$.  Thus it suffices to
treat the case where $n=p^e$.  

We write $\exp_p$ and $\log_p$ for the $p$-adic exponential and
logarithm respectively.  (See, e.g., \cite[IV.1]{KoblitzPNPAZF} for
basic facts on these functions.)  For $y$ sufficiently close to 1
(namely for $|y-1|<|p^{1/(p-1)}|$), we have
$y=\exp_p\left(\log_p(y)\right)$.  Also, it follows from the power
series definition of $\exp_p$, the ultrametric property of $E$, and
the estimate $v_p(n!)\le n/(p-1)$ that if $x\neq0$ and $\ord(x)$ is
sufficiently large (e.g., $\ord(x)>2/(p-1)$ suffices), then
$$\ord\left(1-\exp_p(x)\right)=\ord(x).$$
Now if $e$ is sufficiently large, then $\gamma^{p^e}$ is close to 1,
and $x=\log_p(\gamma^{p^e})=p^e\log_p(\gamma)$ has large valuation and
is not zero, so we may apply the estimate above to deduce that
$$\ord\left(1-\gamma^{p^e}\right)=
\ord\left(1-\exp_p\left(\log_p\left(\gamma^{p^e}\right)\right)\right)=
\ord\left(\log_p\left(\gamma^{p^e}\right)\right)=
\ord\left(p^e\right)+\ord\left(\log_p(\gamma)\right).$$
This last quantity is a linear function of $e$ and thus a linear
function of $\log(p^e)$, and this proves our claim.
\end{proof}

\begin{proof}[Proof of Propositions~\ref{prop:sha-dim} and
\ref{prop:sha-formula}]
We use the leading coefficient part of the BSD conjecture and consider
the $p$-adic valuations of the elements of the formula.  For
simplicity, we first consider the case where $A_0:=\tr_{K/k}(A)=0$ and
then discuss the modifications needed to handle the general case at
the end.

As a first step, we establish that several factors in the BSD formula
do not contribute to the limit in Proposition~\ref{prop:sha-dim}.
More precisely, as $n$ varies, $\Reg(A/K_n)$,
$\tau(A/K_n)$, and $|A(K_n)_{tor}|\cdot|\hat A(K_n)_{tor}|$ are
bounded.  To see that $\Reg(A/K_n)$ is bounded, we note that it is
sensitive to the ground field $\Fqn$ only via the Mordell-Weil group
$A(K_n)/tor$.  In other words, if $A(K_n)/tor=A(K_m)/tor$, then
$\Reg(A/K_n)=\Reg(A/K_m)$.  This follows from the geometric nature of
the definition of $\Reg$ (i.e, its definition in terms of intersection
numbers).  From the Lang-N\'eron theorem on the finite generation of
$A(K\Fqbar)$, it follows that there are only finitely many
possibilities for $A(K_n)/tor$, so only finitely many possibilities
for $\Reg(A/K_n)$.  It also follows that $|A(K_n)_{tor}|$ and
$|\hat A(K_n)_{tor}|$ are bounded.  (Our use of the Lang-N\'eron
theorem here depends on the assumption that $A_0=0$.)  Similarly,
since the orders of the component groups of the fibers of the N\'eron
model of $A$ over $\Fqbar(\CC)$ are bounded, there are only finitely
possibilities for $\tau(A/K_n)$.
Finally, we note that the geometric quantities $\deg(\omega)$, $\dim(A)$,
and $g_\CC$ do not vary with $n$.  

Write $L^*(A/K_n)_p$ for the $p$-part of the rational number
$L^*(A/K_n)$.  Then the BSD formula and the remarks above imply that
\begin{align*}
\lim_{n\to\infty}\frac{\log\left|\sha(A/K_n)[p^\infty]\right|}{\log(q^n)}
&=\lim_{n\to\infty}
 \frac{\log\left(L^*(A/K_n)_pq^{n(\deg(\omega)+\dim(A)(g_\CC-1))}\right)}{\log(q^n)}\\
&=\lim_{n\to\infty}
 \frac{\log\left(L^*(A/K_n)_p\right)}{\log(q^n)}+\deg(\omega)+\dim(A)(g_\CC-1).
\end{align*}

Thus to complete the proof of the existence of the limit in
Proposition~\ref{prop:sha-dim} and the formula of
Proposition~\ref{prop:sha-formula} in the case $A_0=0$, we need only
check that
$$\lim_{n\to\infty}\frac{\log\left(L^*(A/K_n)_p\right)}{\log q^n}=
\sum_{\lambda_i<1}\left(\lambda_i-1\right).$$

Again under the assumption that $A_0=0$, we have
$$L^*(A/K_n)=\sideset{}{'}\prod_i\left(1-(\alpha_i/q)^n\right)$$
where $\prod'_i$ is the product over indices $i$ such that
$(\alpha_i/q)^n\neq 1$.  We view the right hand side as an element of
the number field $F$ introduced above to define the slopes, and we let
$E$ (as in Lemma~\ref{lemma:ords}) be the completion of $F$ at the
chosen prime of $F$ over $p$.  If $\lambda=v(\alpha_i)<1$, then
$$v\left(1-(\alpha_i/q)^n\right)
=v\left((\alpha_i/q)^n\right)
=n(\lambda_i-1),$$
whereas is $\lambda_i>1$, then
$$v\left(1-(\alpha_i/q)^n\right)=v(1)=0.$$
In the intermediate case where $\lambda_i=1$, there are two cases: if
$\alpha_i/q$ is not a root of unity, then Lemma~\ref{lemma:ords}
implies that
$$v\left(1-(\alpha_i/q)^n\right)=O(\log n).$$
If $\alpha_i/q$ is a root of unity, then there are only finitely many
possibilities for $v\left(1-(\alpha_i/q)^n\right)$ with
$(\alpha_i/q)^n\neq1$, and if $(\alpha_i/q)^n=1$, then it does not
contribute to $L^*(A/K_n)$.  Taking the product over $i$, we find that
$$\lim_{n\to\infty}\frac{\log\left(L^*(A/K_n)_p\right)}{\log q^n}=
\sum_{\lambda_i<1}\left(\lambda_i-1\right).$$
This establishes the formula in Proposition~\ref{prop:sha-formula}.

Since the break points of a Newton polygon have integer coordinates,
$\sum_{\lambda_i<1}\left(\lambda_i-1\right)$ is an
integer.  In the case $A_0=0$, we have thus established that the limit
in Proposition~\ref{prop:sha-dim} exists and is an integer,
and we have established the formula in
Proposition~\ref{prop:sha-formula} for the limit, i.e., for
$\dim\sha(A)$.

In case $A_0=\tr_{K/k}(A)$ is non-zero, the $L$-function is more
complicated, the torsion is not uniformly bounded, and we have to be
slightly more careful with the regulator.  Here are the details:  The
Lang-N\'eron theorem says that $A(K\Fqbar)/A_0(\Fqbar)$ is finitely
generated.  This implies that there are only finitely many
possibilities for $A(K_n)/A_0(\Fqn)$ and for the regulator (since
$A(K_n)/tor$ is a quotient of $A(K_n)/A_0(\Fqn)$).  Moreover, 
$$|A(K_n)_{tor}|=\left|\left(A(K_n)/A_0(\Fqn)\right)_{tor}\right|\cdot
|A_0(\Fqn)_{tor}|$$ 
and similarly for $\hat A$.  On the other hand, writing
$$L(A,s)=\frac{P(q^{-s})}{Q(q^{-s})Q(q^{1-s})}
=\frac{\prod_i(1-\alpha_iq^{-s})}
  {\prod_j(1-\beta_jq^{-s})(1-\beta_jq^{1-s})},$$
we have that 
$$L^*(A/K_n)=\frac{\prod_{(\alpha_i/q)^n\neq1}(1-(\alpha_i/q)^n)}
  {\prod_j(1-(\beta_j/q)^n)(1-\beta_j^n)}.$$
The denominator is 
$$q^{-n\dim(A_0)}|A_0(\Fqn)|^2=q^{-n\dim(A_0)}|A_0(\Fqn)|\cdot|\hat A_0(\Fqn)|$$
so the ratio
$$\frac{|A(K_n)_{tor}|\cdot|\hat A(K_n)_{tor}|}
  {\prod_j(1-(\beta_j/q)^n)(1-\beta_j^n)}
=q^{n\dim(A_0)}\left|\left(\vphantom{\hat A_n}A(K_n)/A_0(\Fqn)\right)_{tor}\right|
\cdot\left|\left(\hat A(K_n)/\hat A_0(\Fqn)\right)_{tor}\right|$$
is $q^{n\dim(A_0)}$ times a quantity which is bounded as $n$ varies.  It
then follows that 
$$\lim_{n\to\infty}\frac{\log\left(|A(K_n)_{tor}|\cdot|\hat
    A(K_n)_{tor}|\cdot L^*(A/K_n)_p\right)}{\log q^n}
=\dim(A_0)+\sum_{\lambda_i<1}\left(\lambda_i-1\right).$$
Therefore,
\begin{align*}
\lim_{n\to\infty}\frac{\log|\sha(A/K_n)|}{\log(q^n)}
&=\lim_{n\to\infty}
 \frac{\log\left(|A(K_n)_{tor}|\cdot|\hat
    A(K_n)_{tor}|\cdot L^*(A/K_n)
         q^{n(\deg(\omega)+\dim(A)(g_\CC-1))}\right)}{\log(q^n)}\\
&=
\dim(A_0)+\sum_{\lambda_i<1}\left(\lambda_i-1\right)
+\deg(\omega)+\dim(A)(g_\CC-1).
\end{align*}

This completes the proof of Propositions~\ref{prop:sha-dim} and
\ref{prop:sha-formula}. 
\end{proof}

\begin{rems}\label{rems:dim}
\mbox{}
  \begin{enumerate}
  \item In our applications, we will compute $\dim\sha(A)$ directly
    from its definition using crystalline methods.
    Proposition~\ref{prop:sha-formula} suggests that these methods
    will succeed exactly in those situations where one can compute the
    slopes $\lambda_i$, i.e., exactly in the cases where the methods
    of Hindry-Pacheco and Griffon succeed.
  \item We explain why the terminology ``dimension of $\sha(A)$'' is
    reasonable.  If $\Sel(A,p^m)$ denotes the Selmer group for
    multiplication by $p^m$ on $A$, then it is known that the functor
    $\Fqn\mapsto \Sel(A\times_K K\Fqn,p^m)$ from finite extensions of
    $\Fq$ to groups is represented by a group scheme which is an
    extension of an \'etale group scheme by a unipotent connected
    quasi-algebraic group $U[p^m]$, and the dimension of $U[p^m]$ is
    constant for large $m$ \cite{Artin74}.  (One may even replace
    ``finite extensions of $\Fq$'' with ``affine perfect
    $\Fq$-schemes,'' but unfortunately, not with ``general affine
    schemes.'')  Since the order of $A(K\Fqn)/p^mA(K\Fqn)$ is bounded
    for varying $n$, we may detect the dimension of $U[p^m]$ by
    computing the order of $\sha(A\times K\Fqn)[p^m]$ asymptotically
    as $n\to\infty$.  Thus $\dim\sha(A)$ as we have defined it in this
    paper is equal to the dimension of the unipotent quasi-algebraic
    group $U[p^m]$. (Note however that
    $\Fqn\mapsto\sha(A\times_k K\Fqn)[p^\infty]$ is not in general
    represented by a group scheme.)
  \item The formula in Proposition~\ref{prop:sha-formula} for
    $\dim\sha(A)$ is proven using the BSD formula.  Conversely, in case
    $A$ is a Jacobian, Milne \cite[\S7]{Milne75} computes the dimension of
    the group scheme mentioned in the previous remark, and this
    calculation is a key input into his proof of the leading
    coefficient formula of the BSD and Artin-Tate conjectures.  Our
    approach is thus somewhat ahistorical, but it is elementary
    (modulo the BSD conjecture) and completely general.
  \item In the case where $A$ is a Jacobian, the formula of
    Proposition~\ref{prop:sha-formula} is equivalent to the formula of
    Milne for the unipotent group scheme mentioned above, i.e., to the
    last displayed equation in \cite[\S7]{Milne75}.
  \item The proof of Proposition~\ref{prop:sha-formula} suggests that
    $\dim\sha(A)$ can be viewed as an analog of the Iwasawa
    $\mu$-invariant.
  \item If $K_n=\Fqn(\CC)$, $A$ is an abelian variety over $K_1$
    with $\deg(\omega_A)>0$, and we define the ``$p$-Brauer-Siegel ratio
    of $A$'' by
$$\BS_p(A):=\frac{\log (R|\sha(A)|)_p}{\log H(A)}$$
where $(x)_p$ denotes the $p$-part of the rational number $x$, then we
have
$$\lim_{n\to\infty}\BS_p(A/K_n)=\frac{\dim\sha(A)}{\deg(\omega_A)}.$$
This gives an interpretation of $\dim\sha$ in terms of a modified
Brauer-Siegel ratio.
  \end{enumerate}
\end{rems}

In situations where we can control $\tau(A)$, the following
proposition gives a tool to bound the Brauer-Siegel ratio of $A$ from
below.

\begin{prop}\label{prop:sha-BS-bound}
We have
$$\frac{\log\left(|\sha(A)|\Reg(A)\tau(A)\right)}{\log(q)}\ge
\dim\sha(A).$$  
\end{prop}

\begin{proof}
We keep the notation of the proof of
Proposition~\ref{prop:sha-formula}.  In particular, $A_0$ denotes the
$K/k$ trace of $A$.  Using the BSD formula and estimating the
denominator of $L^*(A)$, we have
\begin{align*}
  |\sha(A)|\Reg(A)\tau(A)
&\ge \frac{|\sha(A)|\Reg(A)\tau(A)}
{\left|\left(\vphantom{\hat A_n}A(K_n)/A_0(\Fqn)\right)_{tor}\right|
\cdot\left|\left(\hat A(K_n)/\hat A_0(\Fqn)\right)_{tor}\right|}\\
&= |A_0(\Fqn)|\cdot|\hat
  A_0(\Fqn)|L^*(A)q^{\deg(\omega)+\dim(A)(g_\CC-1)}\\
&\ge q^{\deg(\omega)+\dim(A)(g_\CC-1)+\dim(A_0)-\sum(1-\lambda_i)}\\
&=q^{\dim\sha(A)}
\end{align*}
and this yields the proposition.
\end{proof}

\begin{rem}\label{rem:not-stupid}
  The bound of the Proposition is more subtle than it may seem at
  first: $\dim\sha(A)$ is defined in terms of the asymptotic growth of
  $\sha(A)$ as the ground field grows (i.e., replacing $\Fq$ with
  $\Fqn$), whereas the left hand side of the inequality concerns
  invariants over the given ground field $\Fq$.  In fact, a lower
  bound on the dimension of $\sha(A)$ is not sufficient to give
  non-trivial lower bounds on $\sha(A)$ itself.  (This is related to
  the non-representability of $\sha$ mentioned above.)  For example,
  if $E$ denotes the Legendre curve studied in \cite{Ulmer14a} over
  $K=\F_{p^{2f}}(t^{1/(p^f+1)})$, then \cite[Cor.~10.2]{Ulmer14a}
  shows that $\dim\sha(E)=(p^f-1)/2$, whereas
  \cite[Thm.~1.1]{Ulmer14d} shows that when $f\le2$, $\sha(E)$ is
  trivial.  This example also shows that the second inequality
  displayed above is sharp.
\end{rem}

Next, we state the result which is our main motivation for considering
$\dim\sha(A)$:

\begin{prop}\label{prop:sha-BS}
  Let $A_d$ be a family of abelian varieties over $K$ with
  $H(A_d)\to\infty$.  Assume that $\tau(A_d)=O(H(A_d)^\epsilon)$ for
  all $\epsilon>0$.  Then
$$\liminf_{d\to\infty} \BS(A_d)\ge   
  \liminf_{d\to\infty}\frac{\dim\sha(A_d)}{\deg(\omega_{A_d})}. $$
\end{prop}

\begin{proof}
The hypothesis $\tau(A_d)=O(H(A_d)^\epsilon)$ for all $\epsilon>0$
implies that
  $$\lim_{d\to\infty}\log(\tau(A_d))/\log(H(A_d))=0,$$ 
  so the proposition follows immediately from the estimate 
  of Proposition~\ref{prop:sha-BS-bound}.
\end{proof}

\begin{cor}\label{cor:sha-BS}
If $A_d$ is a family of abelian varieties over $K$ such that
$H(A_d)\to\infty$, then 
$$\liminf_{d\to\infty} \BS(A_d)\ge   
\liminf_{d\to\infty}\frac{\dim\sha(A_d)}{\deg(\omega_{A_d})}$$ 
in any of the following situations:
\begin{enumerate}
\item $\dim(A_d)=1$ for all $n$
\item $A$ is an abelian variety over $K=\Fq(t)$,  $A_d$ is
  the family associated to the Kummer tower,
  and $A$ has semi-stable reduction at
  $t=0$ and $t=\infty$.
\item $A$ is an abelian variety over $K=\Fq(t)$,  $A_d$ is
  the family associated to the Artin-Schreier tower,
  and $A$ has semi-stable reduction at
  $t=\infty$.
\item $A$ is an abelian variety over $K=\Fq(E)$,  and $A_d$ is
  the family associated to the division tower.
\item $A$ is an abelian variety over $K=\Fq(t)$,  $A_d$ is
  the family associated to the $PGL_2$ tower,
  and $A$ has semi-stable reduction at
  $t=0$ and $t=1$.
\end{enumerate}  
\end{cor}

\begin{proof}
  This is immediate from Lemma~\ref{lemma:tau-ec},
  Corollary~\ref{cor:tau-av}, and Proposition~\ref{prop:sha-BS}.
\end{proof}

\section{Brauer-Siegel ratio and Frobenius}\label{s:Fr}
As a first application of our results on the dimension of $\sha$, we
compute the Brauer-Siegel ratio for sequences of abelian varieties
associated to the Frobenius isogeny.

More precisely, let $E$ be an elliptic curve over the function field
$K=\Fq(\CC)$, and for $n\ge1$, let $E_n$ be the Frobenius base change:
$$E_n:=E^{(p^n)}=E\times_{K}K$$
where the right hand morphism $K\to K$ is the $p^n$-power Frobenius.

Our goal is the following result.

\begin{thm}\label{thm:FrBS}
Assume that $E$ is non-isotrivial.  Then
$$\lim_{n\to\infty}\BS(E_n)=1.$$
\end{thm}

\begin{proof}
First we note that since $E$ is non-isotrivial, $H(E_n)\to\infty$ as
$n\to\infty$.  Indeed, the $j$-invariant of $E$ has a pole, say  of
order $e$, at some place of $K$, so the $j$ invariant of $E_n$ has
a pole of order $ep^n$ at the same place.  This implies that the
degrees of the divisors of one or both of $c_4(E_n)$ and $c_6(E_n)$
also tend to infinity, and this is possible only if
$\deg(\omega_{E_n})$ also tends to infinity.  Since
$H(E_n)=q^{\deg(\omega_{E_n})}$, we have that $H(E_n)\to\infty$.  

Next we note that Proposition~\ref{prop:sha-formula} shows that
$\dim\sha(E_n)-\deg(\omega_{E_n})$ depends only on the $L$-function of
$E_n$, indeed only on the slopes of the $L$-function.  Since $E$ and
$E_n$ are isogenous, they have the same $L$-function, so we have
$$\dim\sha(E_n)-\deg\omega_{E_n}=\dim\sha(E)-\deg\omega_E$$
for all $n$.

Dividing the last displayed equation by $\deg\omega_{E_n}$ and taking
the limit as $n\to\infty$, we get
$$\frac{\dim\sha(E_n)}{\deg\omega_{E_n}}\to1$$
since $\deg(\omega_{E_n})\to\infty$.

Applying part (1) of Corollary~\ref{cor:sha-BS}, we see that
$\liminf_{n\to\infty}BS(E_n)\ge1$.  On the other hand,
$\limsup_{n\to\infty}BS(E_n)\le1$ by
\cite[Cor.~1.13]{HindryPacheco16}, so we find that
$\lim_{n\to\infty}BS(E_n)=1$, as desired.
\end{proof}

\begin{rem}
  The same argument works for an abelian variety $A$ as long as
  $\deg(\omega_{A^{(p^n)}})\to\infty$ with $n$ and
  $\tau(A^{(p^n)})=o(H(A^{(p^n)})$. 
\end{rem}

\begin{rem}
The theorem says that the product $|\sha(E_n)|\Reg(E_n)$ grows with
$n$.  Our earlier results on $p$-descent \cite{Ulmer91} can be
used to show directly that $\sha(E_n)$ grows with $n$.  Full details
require an unilluminating consideration of many cases, so we limit
ourselves to a sketch in the simplest situation.  First, let $V:E^{(p)}\to
E$ be the Verschiebung isogeny, and note that the
Selmer group $\Sel(E^{(p)},p)$ contains $\Sel(E,V)$.  Also, let $L$ be
the (Galois) extension of $K$ obtained by adjoining the $(p-1)$st root
of a Hasse invariant of $E$, and let $G=\gal(L/K)$.
In \cite[Thm.~3.2 and Lemma~1.4]{Ulmer91}, we computed  that 
$$\Sel(E,V)\cong\Hom(J_{\m}/<cusps>,\Z/p\Z)^G$$
where $J_\m$ is the generalized Jacobian of the curve whose function
field is $L$ for a ``modulus'' $\m$ related to the places of bad
and/or supersingular reduction of $E$.  Rosenlicht showed that $J_\m$
is an extension of $J$ by a linear group (see \cite{SerreAGCF}), and
the unipotent part of this group contributes to the ``dimension'' of
$\Sel(E,V)$ and therefore to $\dim\sha(E^{(p)})$.  The contribution is
roughly the number of zeroes (with multiplicity) of the Hasse
invariant, namely $(p-1)\deg(\omega_E)$ which is approximately
$\deg(\omega_{E^{(p)}})-\deg(\omega)$.  Thus we find
$$\dim\sha(E^{(p)})\ge \deg(\omega_{E^{(p)}})-\deg(\omega),$$
in agreement with what we deduced from
Proposition~\ref{prop:sha-formula}. 
\end{rem}

\section{Bounding $\sha$ for a class of Jacobians}\label{s:DPC}
In this section, we review a general method for computing the $p$-part
of the Tate-Shafarevich group of certain Jacobians, generalizing our
previous work \cite{Ulmer14d} on the Legendre elliptic curve.
Although these methods suffice to compute the $p$-part of $\sha$ on
the nose, for simplicity we focus just on $\dim\sha$ as this is what
is needed to bound the Brauer-Siegel ratio from below.

\subsection{Jacobians related to products of curves}
Let $k$ be the finite field $\Fq$ of characteristic $p$ with $q$
elements.  Let $\CC$ and $\DD$ be curves over $k$, and let
$\SS=\CC\times_k\DD$.  Suppose that $\Delta$ is a group of
$k$-automorphisms of $\SS$ with order prime to $p$ and such that
$$\Delta\subset\aut_k(\CC)\times\aut_k(\DD)\subset\aut_k(\SS).$$

Suppose that the quotient $\SS/\Delta$ is birational to a smooth,
projective surface $\XX$ over $k$ and that $\XX$ is equipped with a
surjective and generically smooth morphism $\pi:\XX\to C$ where $C$ is
a smooth projective curve over $k$.  Let $K=k(C)$ and let $X$ be the
generic fiber of $\pi$, a smooth projective curve over $K$.  We assume
that $X$ has a $K$-rational point.  (A vast supply of such data is
given in \cite{Berger08} and \cite{Ulmer13a}.)

Let $J$ be the Jacobian of $X$.  We write $\Br(\XX)$ for the
cohomological Brauer group of $\XX$:  $\Br(\XX)=H^2(\XX,\G_m)$.  

\begin{prop}\label{prop:Br-DPC}\mbox{}
  \begin{enumerate}
  \item $\sha(J_X)$ and $\Br(\XX)$ are finite groups
\item There is a canonical isomorphism $\sha(J_X)\cong\Br(\XX)$.
\item There is a canonical isomorphism
$$\Br(\XX)[p^\infty]\cong\left(\Br(\SS)[p^\infty]\right)^\Delta.$$
  \end{enumerate}
\end{prop}

\begin{proof}
  In substance, parts (2) and (3) are due to Grothendieck \cite{Grothendieck68iii}
  and part (1) is due to Tate \cite{Tate66a}.  The details to deduce
  the statements here are given in \cite[\S4]{Ulmer14d}.
\end{proof}

\subsection{Brauer group of a product of curves}
We keep the notation of the preceding subsection.  In addition, let
$W=W(k)$ be the ring of Witt vectors over $k$ with Frobenius
endomorphism $\sigma$.  We write $H^1(\CC)$ for the crystalline
cohomology $H^1_{crys}(\CC/W)$ and similarly for $H^1(\DD)$.  These are
modules over the Dieudonn\'e ring $A=W\{F,V\}$, which is the
non-commutative polynomial ring generated over $W$ by symbols $F$ and
$V$ with relations $FV=VF=p$, $F\alpha=\sigma(\alpha)F$, and
$\alpha V=V\sigma(\alpha)$ for all $\alpha\in W$.

The following crystalline calculation of the $p$ part of the Brauer
group of $\SS$ is originally due to Dummigan (with additional
hypotheses) using results of Milne, and is proven in general in
\cite[\S10]{Ulmer14d}.

\begin{prop}\label{prop:Br-crys}
There is a canonical isomorphism
$$\Br(\SS)[p^n]\cong\frac{\Hom_A\left(H^1(\CC)/p^n,H^1(\DD)/p^n\right)}
{\Hom_A\left(H^1(\CC),H^1(\DD)\right)/p^n}$$
which is compatible with the actions of $\Delta$ on both sides.
\end{prop}

Here $\Hom_A$ denotes $W$-linear homomorphisms which commute with
$F$ and $V$.

Propositions \ref{prop:Br-DPC} and \ref{prop:Br-crys} give us a
powerful tool for bounding $\dim\sha(J)$ from below.  Recall that this
means bounding the growth of the order of $\sha(J)$ as we extend the
ground field from $\Fq$ to $\F_{q^\nu}$.  The denominator on the right
hand side of the displayed equation in \ref{prop:Br-crys} is known to
be bounded as $\nu$ varies (a fact we will see explicitly in
Section~\ref{s:p-adic} for the examples we consider), so we have:

\begin{cor}\label{cor:dim-sha} For all sufficiently large $n$,
$$\dim\sha(J)=\dim\Hom_A\left(H^1(\CC)/p^n,H^1(\DD)/p^n\right)^\Delta$$
\end{cor}

Here the $\dim$ on the right hand side is defined analogously to that
on the left:
\begin{multline*}
\dim\Hom_A\left(H^1(\CC)/p^n,H^1(\DD)/p^n\right)^\Delta\\
:=\lim_{\nu\to\infty}
\frac{\log|\Hom_A\left(H^1(\CC\times_k\F_{q^\nu})/p^n,
H^1(\DD\times_k\F_{q^\nu})/p^n\right)^\Delta|}{\log(q^\nu)}.
\end{multline*}

Computing the cardinality of the numerator on the right amounts to an
interesting exercise in $p$-linear algebra, at least for certain
curves $\CC$ and $\DD$.  We carry out these exercises in
Section~\ref{s:p-adic}.

\section{Cohomology of Fermat curves}\label{s:Fermat}
We review some well-known result on the cohomology of Fermat curves.

As usual, let $k=\Fq$ be the finite field of cardinality $q$ and
characteristic $p$.  We write $\kbar$ for the algebraic closure of
$k$.  For a positive integer $d$ relatively prime to $p$, let $F_d$ be
the smooth projective curve over $k$ given by
$$x_0^d+x_1^d+x_2^d=0.$$

We write $\mu_d$ for the group of $d$-th roots of unity in $\kbar$.
There is an evident action of $\mu_d^3$ on $F_d\times_k\kbar$ under
which $(\zeta_i)\in\mu_d^3$ acts via $x_i\mapsto\zeta_i x_i$, and the
diagonal $(\zeta_0=\zeta_1=\zeta_2$) acts trivially, so we have
$G:=\mu_d^3/\mu_d\subset\aut_\kbar(F_d)$.

Let 
$$A=\left\{(a_0,a_1,a_2)|\sum a_i=0\right\}\subset(\Z/d\Z)^3.$$
Abusively writing $\zeta$ both for a root of unity in $\kbar$ and for
its Teichm\"uller lift to the Witt vectors $W(\kbar)$, we may identify
$A$ with the character group $\Hom(G,W(\kbar)^\times)$.  Let 
$$A'=\left\{(a_i)\in A|a_i\neq0, i=0,1,2\right\}.$$
Given
$(a_0,a_1,a_2)\in A$, let $\<a_i/d\>$ be the fractional part of
$\tilde a_i/d$, where $\tilde a_i$ is any representative in $\Z$ of
the class $a_i$.
Define subsets $A_0$ and $A_1$ as follows:
$$A_0=\left\{(a_i)\in A'|
\<\frac{a_0}d\>+\<\frac{a_1}d\>+\<\frac{a_2}d\>=2\right\}$$
and
$$A_1=\left\{(a_i)\in A'|
\<\frac{a_0}d\>+\<\frac{a_1}d\>+\<\frac{a_2}d\>=1\right\}$$
It is a simple exercise to see that $A'$ is the disjoint union of
$A_0$ and $A_1$.  Let $\< p\>$ be the subgroup of
$\Q^\times$ generated by $p$.  Then $\< p\>$ acts on $A'$
coordinatewise: $p(a_0,a_1,a_2)=(pa_0,pa_1,pa_2)$.

Let $H=H^1_{crys}(F_d/W(k))$ be the crystalline cohomology of $F_d$
equipped with its action of the $p$-power Frobenius $F$ and
Verschiebung $V$.  Then $\Hbar:=H\tensor_{W(k)}W(\kbar)$ inherits an action
of $G$.

The following summarizes the main results on $H$.  The argument in
\cite[\S6]{Dummigan95}, stated in the special case where $d=q+1$, works
for general $d$ prime to $p$.

\begin{prop}\label{prop:Fermat-cohom}
There is $W$-basis $\{e_a\}$ of $H$ indexed by $a\in A'$ with the
following properties:
  \begin{enumerate}
  \item $F(e_a)=c_ae_{pa}$ where $c_a\in W(k)$ and
$$\ord_p(c_a)=\begin{cases}
0&\text{if $a\in A_0$}\\
1&\text{if $a\in A_1$}
\end{cases}$$
\item For $(\zeta_i)\in G$ and $a\in A'$,
$$(\zeta_i)e_a=a(\zeta_i)e_a=\zeta_0^{a_0}\zeta_1^{a_1}\zeta_2^{a_2}e_a$$
(an equality in $\Hbar$).
\end{enumerate}
\end{prop}

\subsection{A remark on twists}
It is sometimes convenient to work with a different model of the
Fermat curve, namely
$$F_d':\qquad y_0^d+y_1^d=y_2^d.$$
This is a twist of $F_d$ in the sense that they $F_d$ and $F'_d$ become
isomorphic over $\kbar$ via
$$(x_0,x_1,x_2)\mapsto (y_0,y_1,\epsilon y_2)$$
where $\epsilon$ is a $d$-th root of -1.  It follows that
Proposition~\ref{prop:Fermat-cohom} holds for $F_d'$ as well, with
possibly different constants $c_a$ which nevertheless continue to satisfy the
valuation formula in part (1).

\subsection{A remark on quotients}\label{ss:F-quotients}
If $\CC$ is the quotient of $F_d$ by a subgroup of $G'\subset G$, then the
crystalline cohomology of $\CC$ can be identified with the
$W$-submodule of $H$ generated by the $e_a$ whose indices $a$ are
trivial on $G'$.

For example, the hyperelliptic curve
$$\CC_{2,d}:\qquad y^2=x^d+1$$
is the quotient of $F_{2d}'$ by a subgroup of $G$ isomorphic to
$\mu_d\times\mu_2$.  (If $d$ is even, it is also a quotient of $F'_d$,
but it is more convenient to have a uniform statement.)

More generally, the superelliptic curve
$$\CC_{r,d}:\qquad y^r=x^d+1$$
is the quotient of $F'_{rd}$ by a subgroup of $G$ isomorphic to
$\mu_d\times\mu_r$  

The crystalline cohomology $H^1_{crys}(\CC_{r,d}/W(k))$ can then be
identified with the $W$-submodule of $H^1_{crys}(F'_{rd}/W(k))$
generated by the $e_a$ where $a$ has the form
$$a=(a_0,a_1,a_2)=(ir,-ir-jd,jd)\quad0<i<d,\quad0<j<r,\quad
ir+jd\not\equiv0\pmod {rd}.$$
The set $I$ of such indices has cardinality
$(r-1)(d-1)-\gcd(r,d)+1$, and it is the disjoint union $I=I_0\cup I_1$
where
$$I_0=I\cap A_0\cong\left\{(i,j)|0<i<d,\ 0<j<r,\ ir+jd>rd\right\}$$
and
$$I_1=I\cap A_1\cong\left\{(i,j)|0<i<d,\ 0<j<r,\ ir+jd<rd\right\}.$$

In the case where $r=2$ we may further simplify this to 
$$I_0\cong\left\{i|d/2<i<d\right\}$$
and
$$I_1\cong\left\{i|0<i<d/2\right\}.$$

These sets, with their action of $\<p\>$, will play a key role in the
$p$-adic exercises that compute $\dim\sha$ for the Jacobians
introduced in Section~\ref{s:DPC}.

% \subsection{A twisted quotient}\label{ss:F-twist-quotient}
% We consider a third family of Fermat quotients which will play a role in
% Section~\ref{s:Legendre}.  The curve
% $$\DD_d:\qquad 2xy^d=x^2-1$$
% over $\Fq$ with $2d$ prime to $q$ is isomorphic (over an extension of
% $\Fq$) to the quotient of $F_{2d}$ by a subgroup of $G$ isomorphic to
% $\mu_{2d}$.  Explicitly, if $\epsilon^d=1/2$, then the map
% $$(y_0,y_1,y_2)\mapsto \left(x=(y_2/y_1)^d,y=\epsilon
%   y_0^2/(y_1y_2)\right)$$
% presents $\DD_d$ as the quotient of $F'_{2d}$ by the subgroup of $G$
% generated by $(1,\zeta,\zeta^{-1})$.    It carries an action of $\mu_{2d}$
% (namely $(x,y)\mapsto(\zeta^dx,\zeta y)$) and its cohomology can be
% identified with the $W$-submodule of $H^1_{crys}(F_{2d}/W(k))$
% generated by the $e_a$ where $a$ has the form
% $$a=(-2i,i,i)=i(-2,1,1)\quad i\in\Z/2d\Z, i\not\in\{0,d\}.$$
% The set $I$ of such indices is the disjoint union $I=I_0\cup I_1$ where
% $$I_0=I\cap A_0=\{ i|d<i<2d\}$$
% and
% $$I_1=I\cap A_1=\{ i|0<i<d\}.$$

\section{$p$-adic exercises}\label{s:p-adic}
In this section, we carry out the exercises in semi-linear algebra
needed to compute the dimension of $\sha$ for several families of
abelian varieties.

Let $p$ be a prime and let $\Fq$ be the field of cardinality $q$ and
characteristic $p$.  Let $W=W(\Fq)$ be the Witt vectors over $\Fq$,
and let $W_n=W/p^n$.  Write $\sigma$ for the $p$-power Witt-vector
Frobenius.  For a positive integer $\nu$, we write $\F_{q^\nu}$ for
the field of $q^\nu$ elements, $W_\nu=W(\F_{q^\nu})$ for the
corresponding Witt ring, and $W_{n,\nu}$ for $W_\nu/p^n$.

Let $A=W\{F,V\}$ be the Dieudonn\'e ring of non-commutative
polynomials in $F$ and $V$ with relations $FV=VF=p$,
$F\alpha=\sigma(\alpha)F$, and $\alpha V=V\sigma(\alpha)$ for
$\alpha\in W$.  Also, let $A_\nu$ be the ring $W_\nu\{F,V\}$ with
analogous relations.

Let $\< p\>$ be the cyclic subgroup of $\Q^\times$ generated
by $p$.

\subsection{Data}\label{ss:data}
Fix a finite set $I$ equipped with an action of $\< p\>$,
which we write multiplicatively: $i\mapsto pi$.  (In the applications
below, $I$ will typically be a subset of $\Z/d\Z$ for some $d$ not
divisible by $p$.)  Let $M$ be the free $W$-module with basis indexed
by $I$:
$$M:=\bigoplus_{i\in I}We_i.$$
Write $I$ as a disjoint union $I=I_0\cup I_1$ and choose elements
$c_i\in W$ such that 
$$\ord(c_i)=\begin{cases}
0&\text{if $i\in I_0$}\\
1&\text{if $i\in I_1$.}
\end{cases}$$
Define a $\sigma$-semilinear map $F:M\to M$ by setting
$$F(e_i)=c_ie_{pi}$$
and a $\sigma^{-1}$-semilinear map $V:M\to M$ by setting
$$V(e_i)=\frac p{\sigma^{-1}(c_{i/p})} e_{i/p}.$$
These definitions give $M$ the structure of an $A$-module, and there
is an induced $A$-module structure on $M_n:=M\tensor_W W_n$.  Parallel
definitions make $M_\nu:=M\tensor_WW_\nu$ and
$M_{n,\nu}:=M\tensor_WW_{n,\nu}$ into $A_\nu$-modules.

Fix another finite set $J$ equipped with an action of $\<
p\>$, write $J$ as a disjoint union $J=J_0\cup J_1$, 
and choose elements $d_j\in W$ with 
$$\ord(d_j)=\begin{cases}
0&\text{if $j\in J_0$}\\
1&\text{if $j\in J_1$.}
\end{cases}$$
Define
$$N:=\bigoplus_{j\in J}Wf_j,$$
with semilinear maps $F:N\to N$ and $V:N\to N$ defined by
$$F(f_j)=d_jf_{pj}$$
and 
$$V(f_j)=\frac p{\sigma^{-1}(d_{j/p})} f_{j/p}.$$
Then $N$ and $N_n:=N\tensor_WW_n$ are $A$-modules, and parallel
definitions make $N_\nu:=N\tensor_WW_\nu$ and
$N_{n,\nu}:=N\tensor_WW_{n,\nu}$ into $A_\nu$-modules.

Let $\< p\>$ act on $I\times J$ diagonally, and let $O$ be
the set of orbits of this action.  For an orbit $o\in O$, define
$$d(o):=\min\left(\left|\left((I_0\times J_1)\cap o\right)\right|,
\left|\left((I_1\times J_0)\cap o\right)\right|\right).$$

Consider $\Hom_{W_\nu}(N_\nu,M_\nu)$, a free $W_\nu$-module with basis
$\varphi_{ij}$ defined by 
$$\varphi_{ij}(f_{j'})=\begin{cases}
e_i&\text{if $j'=j$}\\
0&\text{if $j'\neq j$.}
\end{cases}$$
These elements induce elements of 
$$\Hom_{W_\nu}(N_{n,\nu},M_{n,\nu})
=\Hom_{W_\nu}(N_\nu,M_\nu)/p^n$$
which form a basis over $W_{n,\nu}$ and which we abusively also
denote $\varphi_{ij}$.

\subsection{Statement}\label{ss:p-adic-statement}
Our main objects of study in this section are the subgroups
$$H_\nu:=\Hom_{A_\nu}(N_\nu,M_\nu)\subset\Hom_{W_\nu}(N_\nu,M_\nu)$$
and
$$H_{n,\nu}:=\Hom_{A_\nu}(N_{n,\nu},M_{n,\nu})
\subset\Hom_{W_{\nu}}(N_{n,\nu},M_{n,\nu})$$ consisting of
$A_\nu$-module homomorphisms, i.e., homomorphisms $\varphi$ such that
$F\compose\varphi=\varphi\compose F$ and
$V\compose\varphi=\varphi\compose V$.

To state the results, we first decompose the groups of interest into
components indexed by the set of orbits $O$.  For $o\in O$, let
$$\Hom_{W_\nu}(N_\nu,M_\nu)^o:=
\left\{\left.\varphi=\sum_{i,j}\alpha_{i,j}\varphi_{i,j}\right|\alpha_{i,j}=0
\text{ for all }(i,j)\not\in o\right\}$$
and
$$\Hom_{W_\nu}(N_{n,\nu},M_{n,\nu})^o:=
\left\{\left.\varphi=\sum_{i,j}\alpha_{i,j}\varphi_{i,j}\right|\alpha_{i,j}=0
\text{ for all }(i,j)\not\in o\right\}.$$
We define
$$H_{\nu}^o:=H_\nu\cap\Hom_{W_\nu}(N_\nu,M_\nu)^o$$
and
$$H_{n,\nu}^o:=H_{n,\nu}\cap\Hom_{W_\nu}(N_{n,\nu},M_{n,\nu})^o.$$

Here is the main result of this section:

\begin{thm}\label{thm:dim}
  \mbox{}
  \begin{enumerate}
  \item $H_{\nu}=\oplus_{o\in O} H_{\nu}^o$ and
    $H_{n,\nu}=\oplus_{o\in O} H_{n,\nu}^o$.
  \item $\left|H_\nu^o/p^n\right|$ is at most $p^{n|o|}$ and in
    particular is bounded independently of $\nu$.
\item For all sufficiently large $n$,
$$\lim_{\nu\to\infty}\frac{\log\left|H_{n,\nu}^o\right|}{\log(q^\nu)}
=d(o).$$
  \end{enumerate}
\end{thm}

\begin{proof}
Let 
$$\varphi=\sum_{(i,j)\in I\times J}\alpha_{i,j}\varphi_{i,j}$$
be a typical element
of $\Hom_{W_\nu}(N_\nu,M_\nu)$ (with $\alpha_{i,j}\in W_\nu$) or
$\Hom_{W_\nu}(N_{n,\nu},M_{n,\nu})$ (with $\alpha_{i,j}\in
W_{n,\nu}$).  Then a straightforward calculation shows that
$F\compose\varphi=\varphi\compose F$ if and only if
\begin{equation}\label{eq:Fphi}
c_i\sigma(\alpha_{i,j})=d_j\alpha_{p(i,j)}\quad
\text{for all }(i,j)\in I\times J,
\end{equation}
and $V\compose\varphi=\varphi\compose V$ if and only if
\begin{equation}\label{eq:Vphi}
\left(\frac p{d_j}\right)\sigma(\alpha_{i,j})
=\left(\frac p{c_i}\right)\alpha_{p(i,j)}\quad
\text{for all }(i,j)\in I\times J.
\end{equation}
Defining
$$\varphi^o=\sum_{(i,j)\in o}\alpha_{i,j}\varphi_{i,j},$$
it is clear that $\varphi^o\in H_\nu^o$ or $H_{n,\nu}^o$ and that
$\varphi=\sum_{o\in O} \varphi^o$.  This shows that 
$H_{\nu}=\sum_{o\in O} H_{\nu}^o$ and
$H_{n,\nu}=\sum_{o\in O} H_{n,\nu}^o$, and it is immediate that
the sums are direct.  This proves part (1) of the theorem.

For part (2), take a a typical element $\varphi^o=\sum_{(i,j)\in
  o}\alpha_{i,j}\varphi_{i,j}$ of $H_\nu^o$.  Since $W_\nu$ is
torsion-free, the conditions \eqref{eq:Fphi} and \eqref{eq:Vphi} are
equivalent, so we focus on \eqref{eq:Fphi}.  Fix a base point
$(i_0,j_0)\in o$ and note that $\alpha_{i_0,j_0}$ determines the other
coefficients $\alpha_{i,j}$ with $(i,j)\in o$ by repeatedly using
\eqref{eq:Fphi}.  Indeed, we have
\begin{align*}
c_{i_0}\sigma(\alpha_{i_0,j_0})&=d_{j_0}\alpha_{p(i_0,j_0)}\\
c_{pi_0}\sigma(c_{i_0})\sigma^2(\alpha_{i_0,j_0})
&=d_{pj_0}\sigma(d_{j_0})\alpha_{p^2(i_0,j_0)}\\
&\vdots\\
c_{p^{|o|-1}i_0}\sigma(c_{p^{|o|-2}i_0})
\cdots\sigma^{|o|-1}(c_{i_0})\sigma^{|o|}(\alpha_{i_0,j_0})
&=d_{p^{|o|-1}j_0}\sigma(d_{p^{|o|-2}j_0})
\cdots\sigma^{|o|-1}(d_{j_0})\alpha_{i_0,j_0}
\end{align*}
Here $|o|$ is the cardinality of $o$ and in the last line we use that
$p^{|o|}(i_0,j_0)=(i_0,j_0)$.  Moreover, $\alpha_{i_0,j_0}$ determines a
solution to \eqref{eq:Fphi} only if the last displayed line holds.
(There may be other integrality conditions, but they are not important
for our argument.)  If the valuations of
$$c_{p^{|o|-1}i_0}\sigma(c_{p^{|o|-2}i_0})\cdots\sigma^{|o|-1}(c_{i_0})
\quad\text{and}\quad
d_{p^{|o|-1}j_0}\sigma(d_{p^{|o|-2}j_0})\cdots\sigma^{|o|-1}(d_{j_0})$$
are distinct, then it is clear that the only solution is
$\alpha_{i_0,j_0}=0$.  On the other hand, if the valuations are the same,
the last equation is equivalent to one of the form
$\sigma^{|o|}(\alpha_{i_0,j_0})=\gamma\alpha_{i_0,j_0}$ where
$\gamma\in W_\nu$ is a unit.  Written in terms of Witt vector
components, this last equation is a polynomial of degree $p^{|o|}$ in
each component of $\alpha_{i_0,j_0}$ (with coefficients given by $\gamma$
and the lower Witt components of $\alpha_{i_0,j_0}$).  Therefore, taking
$\alpha_{i_0,j_0}$ modulo $p^n$, there are at most $p^{n|o|}$ solutions,
and this proves part (2) of the theorem.

We now turn to part (3) of the theorem, which follows from a somewhat
more elaborate version of the calculation of \cite[\S7,
\S10]{Ulmer14d}.  Namely, we fix an orbit $o$ and consider
\eqref{eq:Fphi} and \eqref{eq:Vphi} with $(i,j)\in o$ and
$\alpha_{i,j}\in W_{n,\nu}$.  These are the equations defining
$H_{n,\nu}^o$ as a subset of $\Hom_{W_\nu}(N_{n,\nu},M_{n,\nu})^o$,
and analyzing them will allow us to estimate the size of $H_{n,\nu}^o$.

Fix an orbit $o\in O$ and a base point $(i_0,j_0)\in o$.  We associate a
word $w$ on the alphabet $\{u,l,m\}$ to $o$ as follows:
$w=w_1w_2\cdots w_{|o|}$ where
$$w_\ell=\begin{cases}
u&\text{if $p^{\ell-1}(i_0,j_0)\in I_1\times J_0$}\\
l&\text{if $p^{\ell-1}(i_0,j_0)\in I_0\times J_1$}\\
m&\text{if $p^{\ell-1}(i_0,j_0)\in (I_0\times J_0)\cup(I_1\times J_1)$}
\end{cases}$$
Changing the base point changes $w$ by a cyclic permutation.  Note
that $d(o)$ is the smaller of the number of appearances of $l$ or $u$
in $w$. 

The motivation for these letters is as follows: If $w_\ell=u$, then in
\eqref{eq:Fphi} and \eqref{eq:Vphi} for $(i,j)=p^{\ell-1}(i_0,j_0)$,
$d_j$ is a unit and $p/c_j$ is a unit.  It follows that the two
equations are equivalent and either of them determines
$\alpha_{p^\ell(i_o,j_0)}$ in terms of $\alpha_{p^{\ell-1}(i_o,j_0)}$.
I.e., the ``upper'' $\alpha_{p^\ell(i_o,j_0)}$ is determined by the
``lower'' $\alpha_{p^{\ell-1}(i_o,j_0)}$.  Similarly, if $w_\ell=l$,
the ``lower'' $\alpha_{p^{\ell-1}(i_o,j_0)}$ is determined by the
``upper'' $\alpha_{p^\ell(i_o,j_0)}$.  Finally, if $w_\ell=m$, then
one of \eqref{eq:Fphi} and \eqref{eq:Vphi} implies other and shows
that $\alpha_{p^{\ell-1}(i_o,j_0)}$ and $\alpha_{p^\ell(i_o,j_0)}$
determine each other.  We will use these
observations to eliminate most of the variables in the systems
\eqref{eq:Fphi} and \eqref{eq:Vphi}, and use the simplified system to
estimate the size of $H_{n,\nu}^o$ and prove part (3) of the theorem.

We first deal with three degenerate cases, namely those where $w$ is a
power of $m$, or has no letters $l$, or has no letters $u$.  In all
three cases, $d(o)=0$, so it will suffice to prove that
$|H_{n,\nu}^o|$ is bounded independently of $\nu$.  If $w=m^{|o|}$,
then $\alpha_{i_0,j_0}$ determines all of the
$\alpha_{p^\ell(i_0,j_0)}$, and the system
$(\ref{eq:Fphi}-\ref{eq:Vphi})$ reduces to a single equation
$$\sigma^{|o|}\alpha_{i_0,j_0}=\gamma\alpha_{i_0,j_0}$$
where $\gamma\in W$ us a unit.
This is easily seen to have at most $p^{n|o|}$ solutions for any
$\nu$, as desired.  If $w$ contains no letters $l$, then 
again $\alpha_{i_0,j_0}$ determines all of the
$\alpha_{p^\ell(i_0,j_0)}$, and the system
$(\ref{eq:Fphi}-\ref{eq:Vphi})$ reduces to a single equation
$$p^e\sigma^{|o|}\alpha_{i_0,j_0}=\gamma\alpha_{i_0,j_0}$$
where $e\ge0$ and $\gamma\in W$ is a unit.  (Here $e$ is the number of
appearances of $u$ in $w$.)  If $e=0$, we are in the previous case,
and the equation has at most $p^{n|o|}$ solutions for any $\nu$,
whereas if $e>0$, then this equation is easily seen to have no
solutions.  Finally, if $w$ has no letter $u$, then the system again
reduces to a single equation of the form
$$\sigma^{|o|}\alpha_{i_0,j_0}=\gamma p^e\alpha_{i_0,j_0}$$
which has at most $p^{n|o|}$ solutions for any $\nu$ if $e=0$ and has
no solutions if $e>0$.

For the rest of the argument, we may assume $w$ contains at least one
$u$ and at least one $l$.
Define a function $a:\{0,1,\dots,|o|\}\to\Z$ by setting $a(0)=0$ and
$$a(\ell)=a(\ell-1)+\begin{cases}
1&\text{if $w_\ell=u$}\\
-1&\text{if $w_\ell=l$}\\
0&\text{if $w_\ell=m$}
\end{cases}$$
for $1\le\ell\le|o|$.  

Define the \emph{height} of $o$, denoted $ht(o)$, to be the maximum
value of $a$ minus the minimum value of $a$.  Note that this is
independent of the choice of a base point for $o$.

We divide into two cases depending on whether
$a(|o|)\ge0$ or $a(|o|)\le0$.  

If $a(|o|)\ge0$, we may change base point so that $0=a(0)$ is the
minimum value of $a$ (i.e., $a(\ell)\ge0$ for $0\le\ell\le|o|$) and
$a(|o|-1)>a(|o|)$.  Indeed, start with any base point $(i_0,j_0)$ and
let $\ell_0$ be such that $a(\ell_0)$ is minimum among the $a(\ell)$.
Then replacing $(i_0,j_0)$ with $(i_1,j_1)=p^{\ell_0}(i_0,j_0)$
ensures that $a(\ell)\ge0$ for all $0\le\ell\le|o|$.  If the new word
$w$ ends with $m$ or $u$, we may replace $(i_1,j_1)$ with
$p^{-1}(i_1,j_1)$ without affecting the inequality $a(\ell)\ge0$.
Iterate until the last letter is $l$, thus yielding the desired base
point.  We fix such as base point and denote it $(i_0,j_0)$.

Choose 
$$0=\ell_0<\ell^0<\ell_1<\ell^1\cdots<\ell^{k-1}<\ell_k=|o|$$
such that $a$ is non-decreasing on
$\{\ell_\lambda,\dots,\ell^\lambda\}$ and non-increasing on
$\{\ell^\lambda,\dots,\ell_{\lambda+1}\}$ for $0\le\lambda\le k-1$.
In particular, the $\ell_\lambda$ are the arguments of local minima of
$a$.  Now let
$$\beta_\lambda=\alpha_{p^{\ell_\lambda}(i_0,j_0)}\quad 0\le\lambda\le
k.$$
(Note that $\beta_k=\beta_0$.)  Then the motivating remarks above
about the letters $u$, $l$, $m$ show that the $\beta_\lambda$
determine all the $\alpha_{i,j}$ with $(i,j)\in o$.  The equations
\eqref{eq:Fphi} and \eqref{eq:Vphi} hold if and only if the
$\beta_\lambda$ satisfy the system:
\begin{align}\label{eq:basic}
p^{e_1}\sigma^{\ell_1-\ell_0}\beta_0&=\gamma_1p^{e_2}\beta_1\notag\\
p^{e_3}\sigma^{\ell_2-\ell_1}\beta_1&=\gamma_2p^{e_4}\beta_2\notag\\
&\enskip\vdots\\
p^{e_{2k-1}}\sigma^{\ell_k-\ell_{k-1}}\beta_{k-1}&=\gamma_kp^{e_{2k}}\beta_k\notag
\end{align}
where
\begin{align*}
e_{2\lambda-1}&=\text{\# of appearances of $u$ in the subword
          }w_{\ell_{\lambda-1}+1}\cdots w_{\ell_\lambda}\\
e_{2\lambda}&=\text{\# of appearances of $l$ in the subword
          }w_{\ell_{\lambda-1}+1}\cdots w_{\ell_\lambda}
\end{align*}
and the units $\gamma_\lambda$ are defined by
$$\gamma_\lambda=p^{e_{2\lambda-1}-e_{2\lambda}}
\prod_{\ell=\ell_{\lambda-1}}^{\ell_\lambda-1}
\sigma^{\ell_\lambda-1-\ell}\left(\frac{d_{p^\ell j_0}}{c_{p^\ell i_0}}\right).$$

To recap, the assignment $\varphi\mapsto(\beta_\lambda)$ gives an
injection $H_{n,\nu}^o\into W_{n,\nu}^{k}$ whose image is the set of
solutions to equations \eqref{eq:basic}.  We will finish the proof of
part (3) of the theorem by estimating the number of such solutions.

Since the theorem is an assertion about $H_{n,\nu}^o$ for sufficiently
large $n$, we will assume for the rest of the proof that $n\ge ht(o)$.
Then we have an exact sequence
$$0\to p^{n-ht(o)}H_{n,\nu}^o\to H_{n,\nu}^o\to W_{n-ht(o),\nu}$$
where the right hand map sends a tuple $(\beta_\lambda)$ to the reduction
modulo $p^{n-ht(o)}$ of $\beta_0$.  (Exactness in the middle follows
from the fact that if $\mu\le n-ht(o)$, then we may recover the Witt
components $\beta_\lambda^{(\mu)}$ from $\beta_0$ modulo $p^{n-ht(o)}$
using the equations \eqref{eq:basic} and the fact that $a(\ell)\ge a(0)$
for all $\ell$.)  Moreover, we have
$$\beta_0\equiv
\left(\gamma_1\cdots\gamma_k\right)^{-1}
p^{a(|o|)}\sigma^{|o|}\beta_0\pmod{p^{n-ht(o)}}.$$
It follows that the image of $H_{n,\nu}^o$ in $W_{n-ht(o),\nu}$ has
order at most $p^{|o|(n-ht(o))}$ independently of $\nu$.  (We may even
conclude that it is 0 if $a(|o|)>0$.)  Thus this
image does not contribute to the limit in the theorem, and it will
suffice to bound $p^{n-ht(o)}H_{n,\nu}^o$.  

Note also that if $n'>n\ge ht(o)$, then 
$$p^{n-ht(o)}H_{n,\nu}^o\isoto p^{n'-ht(o)}H_{n',\nu}^o$$
via $(\beta_\lambda)\mapsto (p^{n'-n}\beta_\lambda)$.  Thus we may
assume that $n=ht(o)$ for the rest of the proof.  

To finish the estimation, we ``break'' the circular system
\eqref{eq:basic} into a triangular system, as in
\cite[\S7.6]{Ulmer14d}.  To that end, choose $\lambda$ so that
$a(\ell^\lambda)$ is the maximum of $a$, and note that 
$ht(0)=a(\ell^\lambda)-a(0)=a(\ell^\lambda)$.
Then we have
$$ht(o)=a(\ell^\lambda)=e_1-e_2+\cdots+e_{2\lambda+1}$$
and
\begin{align*}
0=p^{ht(o)}\beta_0&=p^{e_1-e_2+\cdots+e_{2\lambda+1}}\beta_0\\
&=p^{e_3-e_4+\cdots+e_{2\lambda+1}}\sigma^{-\ell_1}\left(\gamma_1\beta_1\right)\\
&\vdots\\
&=p^{e_{2\lambda+1}}\sigma^{-\ell_1}\left(\gamma_1\right)
\cdots\sigma^{-\ell_\lambda}\left(\gamma_\lambda\beta_\lambda\right).
\end{align*}
It follows that $p^{e_{2\lambda+1}}\beta_\lambda=0$.  Using this in
\eqref{eq:basic} and reordering, we obtain a lower-triangular system
\begin{align*}
0&=\gamma_{\lambda+1}p^{e_{2\lambda+2}}\beta_{\lambda+1}\\
0&=-p^{e_{2\lambda+3}}\sigma^{\ell_{\lambda+2}-\ell_{\lambda+1}}\beta_{\lambda+1}
+\gamma_{\lambda+2}p^{e_{2\lambda+4}}\beta_{\lambda+2}\\
&\vdots\\
0&=-p^{e_{2k-1}}\sigma^{\ell_k-\ell_{k-1}}\beta_{k-1}
+\gamma_kp^{e_{2k}}\beta_k\\
0&=-p^{e_1}\sigma^{\ell_1-\ell_0}\beta_0+\gamma_1p^{e_2}\beta_1\\
&\vdots\\
0&=-p^{e_{2\lambda-1}}\sigma^{\ell_\lambda-\ell_{\lambda-1}}\beta_{\lambda-1}
+\gamma_\lambda p^{e_{2\lambda}}\beta_\lambda.
\end{align*}
This system can be rewritten in the form
$$U_1BU_2
\begin{pmatrix}\beta_{\lambda+1}\\ \vdots\\\beta_k\\ \beta_1 \\
  \vdots\\ \beta_\lambda\end{pmatrix}=0$$
where $U_1$ and $U_2$ are diagonal with powers of $\sigma$ and
products of the units $\gamma_i$ in the diagonal entries and where
$$B=\begin{pmatrix}
p^{e_{2\lambda+2}}\\
-p^{e_{2\lambda+3}}&p^{e_{2\lambda+4}}\\
&&\ddots\\
\\
&&-p^{e_{2k-1}}&p^{e_{2k}}\\
&&&-p^{e_1}&p^{e_2}\\
&&&&&\ddots\\
\\
&&&&&-p^{e_{2\lambda-1}}&p^{e_{2\lambda}}
\end{pmatrix}.$$
It follows that the number of solutions to this system is
$$q^{\nu(e_2+e_4+\cdots+e_{2k})}.$$
On the other hand, $e_2+e_4+\cdots+e_{2k}$ is the total number of
appearances of $l$ in the word $w$, and since $a(|o|)\ge0$, $w$ has at
least as many appearances of $u$ as of $l$, so this sum is equal to
$d(o)$.  It follows that $\left|H_{ht(o),\nu}^o\right|=q^{\nu d(o)}$ and
that
$$\lim_{\nu\to\infty}\frac{\log\left|H_{n,\nu}^o\right|}{\log(q^\nu)}=d(o)$$
for any $n\ge ht(o)$.  This completes the proof of part (3) of the
theorem under the hypothesis that $a(|o|)\ge0$.

The proof when $a(|o|)\le0$ is very similar.  Roughly speaking, one
proceeds as above, but with a base point so that $a(|o|)$ is the
minimum of $a$ and with $\beta_k$ playing the role of $\beta_0$.
More precisely, assuming that $w$ has at least one $u$ and at least one
$l$ and that $a(|o|)\le0$, we may choose a base point for $o$ such
that $a(|o|)$ is the minimum value of $a$ and $a(1)>a(0)=0$.  Fix such
a base point, denoted $(i_0,j_0)$, for the rest of the argument.

As before, choose 
$$0=\ell_0<\ell^0<\ell_1<\ell^1\cdots<\ell^{k-1}<\ell_k=|o|$$
such that $a$ is non-decreasing on
$\{\ell_\lambda,\dots,\ell^\lambda\}$ and non-increasing on
$\{\ell^\lambda,\dots,\ell_{\lambda+1}\}$ for $0\le\lambda\le k-1$.
Let
$$\beta_\lambda=\alpha_{p^{\ell_\lambda}(i_0,j_0)}\quad 0\le\lambda\le
k.$$
Then as before, the coefficients $\alpha_{i,j}$ satisfy equations
\eqref{eq:Fphi} and \eqref{eq:Vphi} if and only if the $\beta_\lambda$
satisfy \eqref{eq:basic}.

The same d\'evissage as before shows that it suffices to estimate the
order of $H_{n,\nu}^o$ in the case where $n=ht(o)$.  We make the
circular system \eqref{eq:basic} triangular as follows:  Choose
$\lambda$ so that $a(\ell^\lambda)$ is the maximum of $a$.  Then
$$ht(o)=a(\ell^\lambda)-a(|o|)=e_{2k}-e_{2k-1}+\cdots+e_{2\lambda+2}.$$ 
Therefore,
\begin{align*}
0=p^{ht(o)}\beta_k&=p^{e_{2k}-e_{2k-1}+\cdots+e_{2\lambda+2}}\beta_k\\
&=p^{e_{2k-2}-e_{2k-3}+\cdots+e_{2\lambda+2}}\gamma_k^{-1}
\sigma^{\ell_k-\ell_{k-1}}\left(\beta_{k-1}\right)\\
&\vdots\\
&=p^{e_{2\lambda+2}}\gamma_k^{-1}
\sigma^{\ell_k-\ell_{k-1}}\left(\gamma_{k-1}^{-1}\right)
\sigma^{\ell_k-\ell_{k-2}}\left(\gamma_{k-2}^{-1}\right)
\cdots\sigma^{\ell_k-\ell_{\lambda+1}}\left(\beta_{\lambda+1}\right).
\end{align*}
It follows that $p^{e_{2\lambda+2}}\beta_{\lambda+1}=0$.  Using this
in \eqref{eq:basic} and reordering, we obtain (up to units and powers
of $\sigma$) an upper-triangular system whose diagonal entries are
$p^{e_1}, p^{e_3},\dots,p^{e_{2k-1}}$.

It follows that the number of solutions to \eqref{eq:basic} with
coefficients in $W_{n,\nu}$ (with $n=ht(o)$) is
$q^{\nu(e_1+\cdots+e_{2k-1})}$.  Observing that $a(|o|)\le0$ implies
that $d(o)=e_1+\cdots+e_{2k-1}$, we find that
$\left|H_{ht(o),\nu}^o\right|=q^{\nu d(o)}$ and that
$$\lim_{\nu\to\infty}\frac{\log\left|H_{n,\nu}^o\right|}{\log(q^\nu)}=d(o)$$
for any $n\ge ht(o)$.  This completes the proof of part (3) of the
theorem in the remaining case when $a(|o|)\le0$.
\end{proof}

\section{Equidistribution}\label{s:equi}
We record three equidistribution statements to be used to control the
average behavior of the invariant $d(o)$ from the preceding section.
The first is a consequence of what is proven in
\cite[Thm.~4.1]{Griffon18a}.  The second is a straightforward
``two-variable'' generalization, and the third is a simple corollary
of the first.  We omit the proofs since they are orthogonal to our
main concerns.

\begin{prop}\label{prop:equi1}
\textup{(}Helfgott/Hindry-Pacheco/Griffon\textup{)} Let $A\subset [0,1]$ be an
interval of length $\alpha$.  Let $p$ be a prime number and let $d$
run through positive integers prime to $p$.  Let $\<p\>$ act on $\Z/d\Z$ by
multiplication, and let $O$ be the set of
orbits. Then
$$\lim_{d\to\infty}\frac 1d\sum_{o\in O}
\left|\frac{\left|\left\{a\in o|\<a/d\>\in A\right\}\right|}{|o|}-\alpha\right|=0.$$
\end{prop}

\begin{prop}\label{prop:equi2}
Let $p$ be a prime number, let $r$ be a fixed integer prime to $p$ and
let $d$ run through integers prime to $p$.  Let $\<p\>$ act on
$(\Z/r\Z)\times(\Z/d\Z)$ diagonally, and let $O$ be the set of orbits.  Then
$$\lim_{d\to\infty}\frac 1d\sum_{o\in O}
\left|\frac{\left|\left\{(a,b)\in o|\<a/r\>+\<b/d\><1\right\}\right|}{|o|}
-\frac12\right|=0.$$
\end{prop}

\begin{prop}\label{prop:equi3}
  Let $p$ be a prime number, let $I=\Z/d\Z$ with $d$ prime to $p$
  equipped with the multiplication action of $\<p\>$, and let
  $J=\{0,1\}$ be a two-element set equipped with the non-trivial action of
  $\<p\>$.  Let $\<p\>$ act on $I\times J$ diagonally, and let $O$ be
  the set of orbits.  Then
$$\lim_{d\to\infty}\frac 1{d}\sum_{o\in O}
\left|\frac{\left|\left\{(a,b)\in o|\<a/d\><1/2, b=0\right\}\right|
+\left|\left\{(a,b)\in o|\<a/d\>>1/2, b=1\right\}\right|}{|o|}
-\frac12\right|=0.$$
\end{prop}

\section{Calculations for curves defined by four monomials}
In this section we compute the limit of Brauer-Siegel ratios for a
family of elliptic curves related to the constructions in
\cite{Shioda86} and \cite{Ulmer02}.  We then explain how the same can
be done for families of Jacobians of every genus in every
positive characteristic.
%These are closely related to the
%calculations in the previous section, but it is convenient to give a
%separate treatment.

Throughout, let $k=\Fq$, the finite field of cardinality $q$ and
characteristic $p$, and let $K=k(t)$, the rational function field over
$k$.    

\subsection{The curve of \cite{Ulmer02}}
Let $p$ be a prime number, let $d$ be a positive integer prime to $p$,
and let $E_d$ be the elliptic curve over $K$ defined by
\begin{equation}\label{eq:Annals}
y^2+xy=x^3-t^d
\end{equation}
This family of curves was introduced in \cite{Ulmer02} where it was
shown that $\sha(E_d)$ is finite and the rank of $E_d(K)$ is unbounded
as $d$ varies.  Hindry and Pacheco \cite{HindryPacheco16} computed the
Brauer-Siegel ratio of $E_d$ as $d\to\infty$ by analytic means, i.e.,
by a careful study of the $L$-function of $E_d$.  Here we compute it
via algebraic means, more precisely, through a consideration of
$\dim\sha(E_d)$.

\begin{thm}
  We have
$$\lim_{d\to\infty}\BS(E_d)=1.$$
\end{thm}

\begin{proof}
  Because $E_{pd}=E_d^{(p)}$, Theorem~\ref{thm:FrBS} implies that it
  will suffice to compute the limit as $d$ runs through positive
  integers relatively prime to $p$ and tending to infinity.

  We are going to bound $\BS(E_d)$ from below by estimating
  $\dim\sha(E_d)$.  Since the latter is invariant under extension of
  the ground field, we are free to extend $k$ as needed and will do so
  in the geometric argument below.

  Let $\EE_d$ be the smooth projective surface equipped with a
  relatively minimal morphism $\pi:\EE_d\to\P^1$ whose generic fiber
  is $E_d$.  The procedure for constructing a model $\EE_d$ is
  explained in general in \cite[Lecture 3]{Ulmer11}, and this
  particular example is carried out in detail in
  \cite[\S3]{Ulmer02}. The important thing to know about $\EE_d$ is
  that it is birational to the hypersurface in $\A^3_{(x,y,t)}$
  defined by the equation \eqref{eq:Annals}.

  Using the method of \cite{Shioda86}, it is proven in
  \cite[\S4]{Ulmer02} that $\EE_d$ is birational to the quotient of
  the Fermat surface of degree $d$ by a group of order $d^2$.  It is
  proven in \cite{ShiodaKatsura79} that the Fermat surface of degree
  $d$ is birational to the quotient of the product of two Fermat
  curves of degree $d$ by a group of order $d$. (Here we may need to
  extend $k$ so that it contains the $2d$th roots of unity.)  Putting
  these together, we find that $\EE_d$ is birational to the quotient
  of $F_d\times F_d$ by the group
$$\Delta\subset(\mu_d^3/\mu_d)^2\subset\aut(F_d)\times\aut(F_d)$$
generated by
$$([\zeta^2,\zeta,1],[1,1,1]),
([1,\zeta,1],[\zeta^{3},1,1]),
\text{ and }([1,1,\zeta],[1,1,\zeta])$$
where $\zeta$ is a primitive $d$th root of unity in $k$.  

It follows from Corollary~\ref{cor:dim-sha} that
\begin{equation}\label{eq:Adim}
\dim\sha(E_d)=
\dim\Hom_A\left(H^1(F_d)/p^n,H^1(F_d)/p^n\right)^\Delta
\end{equation}
for all sufficiently large $n$.  Section~\ref{s:Fermat} and
Proposition~\ref{prop:Fermat-cohom} describe the cohomology group
$H^1(F_d)$ with its action of Frobenius.  They show in particular that
the dimension in the last display can be computed by the methods of
Section~\ref{s:p-adic}.

To spell this out, recall that the cohomology of $F_d$ splits into
lines indexed by 
$$A'=\left\{(a_0,a_1,a_2)|a_i\neq0, \sum
  a_i=0\right\}\subset(\Z/d\Z)^3$$ and that $A'$ is the disjoint
union of $A_0$ and $A_1$  as in Section~\ref{s:Fermat}.  The curves
$F_d$ and their cohomology furnish data $M=N=H^1_{crys}(F_d/W(k))$,
$I=J=A'$, and $(c_i,d_j)$ as in Subsection~\ref{ss:data}.

A short calculation reveals that the basis elements $\varphi_{ij}$ which
contribute to the right hand side of \eqref{eq:Adim} are those indexed
by $(i,j)$ of the form
$$(i,j)=(a_0,a_1,a_2,b_0,b_1,b_2)=b_1(-3,6,-3,2,1,-3)$$
where $b_1\in\/d\Z$ is such that $6b_1\neq0$.  In other words,
projection to the $b_1$ coordinate allows us to identify the orbits of
$\<p\>$ on $I\times J$ which contribute to \eqref{eq:Adim} with the
orbits of $\<p\>$ on
$$B=\left\{b\in\Z/d\Z|6b\neq0\right\}.$$

Under this identification, $(i,j)\in I_0\times J_1$ if and only if 
$$0<\left\<\frac bd\right\><1/6$$
and $(i,j)\in I_1\times J_0$ if and only if 
$$5/6<\left\<\frac bd\right\><1$$
where $\<\cdot\>$ denotes the fractional part.  Thus, the invariant
$d(o)$ of Subsection~\ref{ss:data} becomes the following invariant of
orbits of $\<p\>$ on $B$: Setting
$$B_0=\left\{b\in\Z/d\Z|0<\<b/d\><1/6\right\}\text{ and }
B_1=\left\{b\in\Z/d\Z|5/6<\<b/d\><1\right\},$$
we have
$$d(o)=\min\left(|o\cap B_0|,|o\cap B_1|\right).$$

Finally, the equidistribution result Proposition~\ref{prop:equi1} yields that
$$\sum_{o\in O}d(o)=d/6+\epsilon$$
where $\epsilon/d\to0$ as $d\to\infty$, and so
$$\dim\sha(E_d)= d/6+\epsilon.$$

It follows from \cite[\S2]{Ulmer02} that $\deg\omega_{E_d}=\lceil
d/6\rceil$, so by applying Corollary~\ref{cor:sha-BS}, we conclude that 
$$\liminf_{d\to\infty}BS(E_d)\ge1.$$
Taking into account the upper bound \eqref{eq:BS-HP} of Hindry and
Pacheco, we finally conclude that
$$\lim_{d\to\infty}BS(E_d)=1.$$
\end{proof}

\subsection{Other elliptic curves}
The methods employed in the previous subsection can be used to compute
the limiting Brauer-Siegel ratio for several other families of
elliptic curves, namely those defined by equations involving 4
monomials.  This includes the Hessian family studied in
\cite[Ch.~5]{GriffonThesis} and a closely related family introduced by
Davis and Occhipinti \cite{DavisOcchipinti16} and studied in
\cite[Ch.~7]{GriffonThesis}.  We will not give the details here, since no
fundamentally new phenomena arise.

\subsection{Higher genus Jacobians}
For every prime $p$ and every $g>0$, there is a sequence of curves of
genus $g$ over $\Fp(t)$ whose Jacobians are absolutely simple, satisfy
the Birch and Swinnerton-Dyer conjecture, and have unbounded analytic
and algebraic ranks; see \cite[\S7]{Ulmer07b}.  Since these curves are
defined by four monomials, the methods of this paper suffice to
compute the limit of their Brauer-Siegel ratios. In the rest of this
subsection, we explain the details for the main case, namely when $g$
is a positive integer and $p$ is a prime such that
$p\nodiv(2g+2)(2g+1)$.  The other cases are similar and we omit them
in the interest of brevity.

Fix a positive integer $g$, a prime $p$ such that
$p\nodiv(2g+2)(2g+1)$, and a positive integer $d$.  Let $X_d$ be the
smooth, proper curve of genus $g$ over $K=\Fp(t)$ defined by
\begin{equation}\label{eq:Inv}
y^2=x^{2g+2}+x^{2g+1}+t^d
\end{equation}
and let $J_d$ be its Jacobian.

\begin{thm}
  $$\lim_{d\to\infty}\BS(J_d)=1.$$
\end{thm}

\begin{proof}
  Once again, it suffices to restrict to $d$ not divisible by $p$.  We
  will bound $\BS(J_d)$ from below by estimating $\dim\sha(J_d)$ using
  that $X_d$ has a model which is dominated by a product of Fermat
  curves.  As usual, we are free to expand the ground field $\Fp$ and
  we do so as needed below.

Let $\XX_d$ be the smooth projective surface equipped with a relatively
minimal morphism $\pi:\XX_d\to\P^1$ with generic fiber $X_d$.  Again,
what is most important is that $\XX_d$ is birational to the
hypersurface in $\A^3$ defined by the equation~\eqref{eq:Inv}. 

Using the method of \cite{Shioda86} (see also \cite{Ulmer07b}) , one
sees that $\XX_d$ is birational to the quotient of the Fermat surface
of degree $2d$ by a group of order $(2d)^2$, and therefore birational
to the quotient of $F_{2d}\times F_{2d}$ by a group of order $(2d)^3$.
(Here we enlarge $\Fp$ to a finite extension $k$ that contains the
$2d$th roots of unity.)  More precisely, carrying out the procedure of
\cite[\S6]{Ulmer07b} and using \cite{ShiodaKatsura79}, one finds that
$\XX_d$ is birational to the quotient of $F_{2d}\times F_{2d}$ by the group
$$\Delta\subset(\mu_{2d}^3/\mu_{2d})^2\subset\aut(F_{2d})\times\aut(F_{2d})$$
generated by
$$([\zeta^2,1,1],[1,1,1]),
([1,1,1],[1,\zeta^d,1]),
([1,1,1],[\zeta,\zeta^{2g+2},1]),
\text{ and }([1,1,\zeta],[1,1,\zeta])$$
where $\zeta$ is a primitive $2d$th root of unity in $k$.  

It follows from Corollary~\ref{cor:dim-sha} that
\begin{equation}\label{eq:Invdim}
\dim\sha(E_d)=
\dim\Hom_A\left(H^1(F_d)/p^n,H^1(F_d)/p^n\right)^\Delta  
\end{equation}
for all sufficiently large $n$.

As in the previous subsections, the curves $F_{2d}$ and their cohomology
furnish data $M=N=H^1_{crys}(F_{2d}/W(k))$, $I=J=A'$, and $(c_i,d_j)$ as
in Subsection~\ref{ss:data}.

A short calculation reveals that the basis elements $\varphi_{ij}$ which
contribute to the right hand side of \eqref{eq:Invdim} are those
indexed by $(i,j)$ of the form
$$(i,j)=(a_0,a_1,a_2,b_0,b_1,b_2)=
\left(-(4g+4)b,2b,(4g+2)b,d,d-(4g+2)b,(4g+2)b\right)$$
where $b\in\Z/d\Z$ is such that none of the coordinates $a_0,\dots,b_2$
are zero in $\Z/2d\Z$.  (Note that all of the coefficients of $b$ above
are even, so the display gives a well-defined element of $(\Z/2d\Z)^6$
even though $b$ lies in $\Z/d\Z$.)  Thus the relevant orbits of
$\<p\>$ on $I\times J$ can be identified with the orbits of $\<p\>$ on
the subset $B$ of $\Z/d\Z$ where none of the coordinates of $(i,j)$ is 0.

Next we work out conditions on $b$ for the corresponding $(i,j)$ to
lie in $I_0\times J_1$ or $I_1\times J_0$.  One finds that 
$$i=(a_0,a_1,a_2)=(-(4g+4)b,2b,(4g+2)b)$$
lies in $I_0$ if and only if the fractional part $\<b/d\>$ lies in one of
the intervals 
$$\left(\frac{k+1}{2g+2},\frac{k+1}{2g+1}\right),\quad k=0,\dots,2g$$
and $i$ lies in $I_1$ if and only if the fractional part $\<b/d\>$ lies in one of
the intervals 
$$\left(\frac{k}{2g+1},\frac{k+1}{2g+2}\right),\quad k=0,\dots,2g.$$
On the other hand, 
$$j=(b_0,b_1,b_2)=(d,d-(4g+2)b,(4g+2)b)$$
lies in $J_0$ if and only if the fractional part $\<b/d\>$ lies in one of
the intervals 
$$\left(\frac{2\ell+1}{4g+2},\frac{2\ell+2}{4g+2}\right),\quad \ell=0,\dots,2g$$
and $j$ lies in $J_1$ if and only if the fractional part $\<b/d\>$ lies in one of
the intervals 
$$\left(\frac{2\ell}{4g+2},\frac{2\ell+1}{4g+2}\right),\quad \ell=0,\dots,2g.$$

It follows that $(i,j)$ lies in $I_0\times J_1$ if and only if
$$\<b/d\>\in\left(\frac{k+1}{2g+2},\frac{2k+1}{4g+2}\right)$$
with $k=g+1,\dots,2g$ and it lies in $I_1\times J_0$ if and only if
$$\<b/d\>\in\left(\frac{2k+1}{4g+2},\frac{k+1}{2g+2}\right)$$
with $k=0,\dots,g-1$.

The total length of the intervals corresponding to $I_0\times J_1$ is
$$\sum_{k=g+1}^{2g}\left(\frac{2k+1}{4g+2}-\frac{k+1}{2g+2}\right)
=\frac{g}{8g+4}$$
and the total length of the intervals corresponding to $I_1\times J_0$ is
$$\sum_{k=0}^{g-1}\left(\frac{k+1}{2g+2}-\frac{2k+1}{4g+2}\right)
=\frac{g}{8g+4}.$$

Transferring the definition of $d(o)$ to $B$ and applying the
equidistribution result Proposition~\ref{prop:equi1}, we find that
$$\dim\sha(J_d)= \sum_{o}d(o)=\frac{dg}{8g+4}+\epsilon$$
where $\epsilon/d\to0$ as $d\to\infty$.

We pause briefly to consider the case $g=1$.  By \cite{Weil54}, the
Jacobian of $X_d$ is the elliptic curve
$$y^2=x^3-4t^dx+t^d.$$
It is easy to see that the bundle $\omega_d$ attached to $J_d$ has
degree $\lceil d/12\rceil$.  It then follows from our estimation of
$\dim\sha(J_d)$ and Corollary~\ref{cor:sha-BS} that
$\liminf_{d\to\infty}\BS(J_d)\ge1$ and thus, by the Hindry-Pacheco
upper bound \eqref{eq:BS-HP}, that $\lim_{d\to\infty}\BS(J_d)=1$.

To extend this to higher genus, we will give an upper
bound on the degree of $\omega_d$ of the form $dg/(8g+4)+\epsilon$
where $\epsilon/d\to0$ as $d\to\infty$.  More precisely, we will show
that $\deg(\omega_d)=dg/(8g+4)$ for all $d$ divisible by
$(2g+1)(2g+2)$.  For a general $d$, we let 
$$d'=\lcm(d,(2g+1)(2g+2))$$
and apply Lemma~\ref{lemma:deg-omega} to conclude that
$$\deg(\omega_d)\le dg/(8g+4)+2g\left((2g+1)(2g+2)-1\right)$$
which gives the desired estimate.

For $i=1,\dots,g$, let $\omega_i$ be the 1-form $x^{i-1}dx/y$ on $X_d$
over $K$.  These 1-forms are regular and give a basis of
$H^0(X,\Omega^1_{X/K})$.  We will consider their extensions to a
suitable model $\pi:\XX\to\P^1$ of $X$ and use them to compute
$\deg(\omega_d)$.  

In \cite[\S7.7]{Ulmer07b}, a model of $X$ over
$U=\P^1\setminus\{0,\infty\}$ is constructed which is regular and a
Lefschetz pencil, i.e., its singular fibers are irreducible with one
ordinary node each.  It is easy to see that the differentials
$\omega_i$ extend to this model and
$$\sigma:=\omega_1\wedge\cdots\wedge\omega_g$$
defines a nowhere vanishing section of $\omega_d$ over $U$.
To compute $\deg(\omega_d)$ it will thus suffice to compute the order
of vanishing of $\sigma$ at $t=0$ and $t=\infty$. This is where we use
the hypothesis that $d$ is a multiple of $(2g+1)(2g+2)$.

Indeed, if $d=2(2g+1)k$, then the change of
coordinates $x\to t^{2k}x'$, $y\to t^{(2g+1)k}y'$ brings $X$ into the form
$$y^{\prime2}=t^{2k}x^{\prime2g+2}+x^{\prime2g+1}+1$$
which has good reduction at $t=0$.  Moreover, we see that 
$\omega_i=t^{(2i-2g-1)k}\omega'_i$ where $\omega'_i=\frac{x^{\prime
    i-1}dx'}{y'}$, and that the $\omega'_i$ have linearly independent
reductions at $t=0$.  This shows that $\sigma$ has a pole at $t=0$ of
order
$$\sum_{i=1}^{g}\frac{(2g+1-2i)d}{2(2g+1)}.$$

Similarly, when $d=(2g+2)\ell$, the
change of coordinates $x\to t^{2\ell}x$, $y\to t^{(2g+2)\ell}y$ brings $X$ into
the form
$$y^2=x^{2g+2}+t^{-\ell}x^{2g+1}+1$$
which has good reduction at $t=\infty$.  Moreover, we see that
$\omega_i=t^{(i-g-1)\ell}\omega'_i$ where
$\omega'_i=\frac{x^{\prime i-1}dx'}{y'}$, and that the $\omega'_i$ have
linearly independent reductions at $t=\infty$.  This shows that $\sigma$
has a zero at $t=\infty$ of order
$$\sum_{i=1}^{g}\frac{(g+1-i)d}{{2g+2}}.$$

A short computation then shows that $\deg(\omega_d)$ is $dg/(8g+4)$.

Note that these calculations also show that $J_d$ has good reduction
at $t=0$ and $t=\infty$ when $d$ is divisible by $(2g+1)(2g+2)$.
Using Section~\ref{ss:tau-tame}, these reduction results imply that
$\tau(J_d)=O(H(J_d)^\epsilon)$ for all $\epsilon>0$.  Then
Proposition~\ref{prop:sha-BS} shows that
$$\liminf_{d\to\infty}\BS(J_d)\ge\liminf_{d\to\infty}
\frac{\dim(\sha(J_d))}{\deg(\omega_{J_d})}\ge1.$$ 
Taking into account the upper bound \eqref{eq:BS-HP} of Hindry and
Pacheco, we finally conclude that
$$\lim_{d\to\infty}BS(J_d)=1.$$
\end{proof}

\section{Calculations for Jacobians related to Berger's
  construction}\label{s:Legendre}
In this section we compute the limiting Brauer-Siegel ratio for some
families of curves related to the construction in \cite{Berger08} and
\cite{Ulmer13a}.

Throughout, let $k=\Fq$, the finite field of cardinality $q$ and
characteristic $p$, and let $K=k(t)$, the rational function field over
$k$.

\subsection{The Legendre curve}\label{ss:Legendre}
Assume that $p>2$, let $d$ be a positive integer, and let
$E_d$ be the elliptic curve over $K$ defined by
\begin{equation}\label{eq:Leg}
y^2=x(x+1)(x+t^d).
\end{equation}
This family of curves has been studied extensively, in particular in
\cite{Ulmer14a, ConceicaoHallUlmer14, Ulmer14d} and
\cite[Ch.~4]{GriffonThesis}.  In the latter, the limit of the Brauer-Siegel
ratio of $E_d$ as $d\to\infty$ was computed by analytic means, i.e.,
by a careful study of the $L$-function of $E_d$.  Here we compute it via
algebraic means, more precisely, through a consideration of
$\dim\sha(E_d)$.

\begin{thm}
  We have
$$\lim_{d\to\infty}\BS(E_d)=1.$$
\end{thm}

\begin{proof}
  As usual, it suffices to consider values of $d$ not divisible by
  $p$.

  Let $\EE_d$ be the smooth projective surface equipped with a
  relatively minimal morphism $\pi:\EE_d\to\P^1$ whose generic fiber
  is $E_d$.  This is constructed in \cite{Ulmer14a} (under the
  simplifying hypothesis that $d$ is even, but the odd case is
  similar).  The main thing we need to know about $\EE_d$ is that it
  is birational to the hypersurface in $\A^3_{(x,y,t)}$ defined by the
  equation \eqref{eq:Leg}.

Let $\CC_d$ be the curve with affine equation
$$x^2=z^d+1$$
and let $\DD_d$ be the curve with affine equation
$$y^2=w^d+1.$$
Both curves admit an evident action of $\Delta=\mu_2\times\mu_d$ (over
$\kbar$).  Let $\Delta$ act ``anti-diagonally'' on $\CC_d\times\DD_d$:
$$(\zeta_2,\zeta_d)(x,z,y,w)
=\left(\zeta_2x,\zeta_dz,\zeta_2^{-1}y,\zeta_d^{-1}w\right).$$
Our first main claim is that $\EE_d$ is birational to the quotient
$\CC_d\times\DD_d/\Delta$ via the map
$$(x,z,y,w)\mapsto \left(x=z^d,y=z^dxy,t=wz\right).$$
Indeed, it is evident that this defines a dominant rational map from
$\CC_d\times\DD_d$ to $\EE_d$ which factors through the
quotient by $\Delta$.  Degree considerations then show that the induced
map has degree 1, i.e., is a birational isomorphism.  

We are thus in position to apply the machinery of Section~\ref{s:DPC}.
In particular, it follows from Corollary~\ref{cor:dim-sha} that
\begin{equation}\label{eq:Ldim}
\dim\sha(E_d)=
\dim\Hom_A\left(H^1(\CC_d)/p^n,H^1(\DD_d)/p^n\right)^\Delta
\end{equation}
for all sufficiently large $n$.
Subsection~\ref{ss:F-quotients} and
Proposition~\ref{prop:Fermat-cohom} describe the cohomology groups
$H^1(\CC_d)$ and $H^1(\DD_d)$ with their actions of Frobenius.  They
show in particular, that the dimension in the last display can be
computed by the methods of Section~\ref{s:p-adic}.

To spell this out, let 
$$I=J=\Z/d\Z\setminus\{0,d/2\text{ (if $d$ is even)}\},$$ 
decomposed as $I_0=J_0=\{i|d/2<i<d\}$ and
$I_1=J_1=\{i|0<i<d/2\}$.  Section~\ref{s:Fermat} shows that the
crystalline cohomology groups $H^1(\CC_d)$ and $H^1(\DD_d)$ with their
action of Frobenius furnish
data ($M$, $N$, $I$, $J$, $c_i$, $d_j$) as in
Subsection~\ref{ss:data}, as well as the invariant $d(o)$ for each
orbit $o$ of $\<p\>$ on $I\times J$.  

Since $\Delta$ acts anti-diagonally, the orbits that contribute to the
right hand side of equation~\ref{eq:Ldim} are those whose elements
$(i,j)$ satisfy $j=-i$.  Write $O^\Delta$ for the set of such orbits.
Applying Theorem~\ref{thm:dim}, we conclude that
\begin{equation}\label{eq:sha-d}
\dim\sha(E_d)=\sum_{o\in O^\Delta}d(o).  
\end{equation}

We may identify the orbits in $O^\Delta$ with the orbits of $\<p\>$ on
$I$ via the projection $\pi_I:I\times J\to I$.  Also, since
$(i,-i)\in I_0\times J_1$ if and only if $i\in I_0$, and
$(i,-i)\in I_1\times J_0$ if and only if $i\in I_1$, we have
$$d(o)=\min(|\pi_I(o)\cap I_0|,|\pi_I(o)\cap I_1|).$$

Thus the sum on the right hand side of \eqref{eq:sha-d} becomes a sum
over orbits of $\<p\>$ on $I$, and the invariant $d(o)$ is 
described ``on average'' in Section~\ref{s:equi}.  In particular, the
equidistribution result Proposition~\ref{prop:equi1} implies that
$$\dim\sha(E_d)=\sum_{o\in O^\Delta}d(o)=d/2+\epsilon_d$$
where $\epsilon_d/d\to0$ as $d\to\infty$.

Since $\deg(\omega_{E_d})=\lceil d/2\rceil$ (e.g., by
\cite[Lemma~7.1]{Ulmer14a}), Corollary~\ref{cor:sha-BS} implies
that
$$\liminf_{d\to\infty}BS(E_d)
\ge\liminf_{d\to\infty}\frac{\dim\sha(E_d)}{\deg(\omega_{E_d})}=1.$$ 
Taking into account the upper bound \eqref{eq:BS-HP} of Hindry and
Pacheco, we conclude that
$$\lim_{d\to\infty}BS(J_d)=1.$$
\end{proof}
 
\subsection{Other elliptic curves}
The methods employed in the previous subsection can be used to compute
the limiting Brauer-Siegel ratio for several other families of
elliptic curves, namely those coming from Berger's construction where
the dominating curves are related to Fermat curves.  This is the case
in particular for the universal curve over $X_1(4)$ studied in
\cite[Ch.~6]{GriffonThesis} and the curve ``$B_{1/2,d}$'' introduced
in \cite[\S4]{Berger08} and studied in \cite[Ch.~8]{GriffonThesis}.
We will not give the details here, since no fundamentally new
phenomena arise.

\subsection{Higher dimensional Jacobians}\label{ss:AIM}
Let $p$ be a prime number, let $q$ be a power of $p$, and let $k=\Fq$.
Let $r$ and $d$ be integers relatively prime to $p$.  Let $X=X_{r,d}$
be the smooth projective curve over $K=k(t)$ associated to the
equation
\begin{equation}\label{eq:AIM}
y^r=x^{r-1}(x+1)(x+t^d).
\end{equation}
This is a curve of genus $r-1$, and the case $r=2$ is the Legendre
curve of Subsection~\ref{ss:Legendre}.  Let $J=J_{r,d}$ be the Jacobian
of $X$.  This family of Jacobians was studied in \cite{AIMgroup},
where among other things it was proven that $\sha(J_{r,d})$ is finite
for all $p$, $q$, $r$, and $d$ as above.  Here we will compute the
limiting Brauer-Siegel ratio for fixed $q$ and $r$ as $d\to\infty$.

\begin{thm}\label{thm:BS-AIM} 
For all $q$ and $r$ as above,
$$\lim_{\substack{d\to\infty\\(p,d)=1}}\BS(J_{r,d})=1.$$
\end{thm}

Here the limit is through integers prime to $p$.  It would be possible
to include those $d$ divisible by $p$ using a straightforward
generalization of the ideas in Section~\ref{s:Fr}, but will not do
that here.

\begin{proof}
  Since $r$ will be fixed throughout, we omit it from the notation.
  Let $\XX_d$ be the smooth projective surface equipped with a
  relatively minimal morphism $\pi:\XX_d\to\P^1$ whose generic fiber
  is $X_d$.  This is constructed in \cite[\S3.1]{AIMgroup}.
  The important thing to know about $\XX_d$ is that it
  is birational to the hypersurface in $\A^3_{(x,y,t)}$ defined by the
  equation \eqref{eq:AIM}.

Let $\CC_d$ be the curve with affine equation
$$x^r=z^d+1$$
and let $\DD_d$ be the curve with affine equation
$$y^r=w^d+1.$$
Both curves admit an evident action of $\Delta=\mu_r\times\mu_d$ (over
$\kbar$).  Let $\Delta$ act ``anti-diagonally'' on $\CC_d\times\DD_d$:
$$(\zeta_r,\zeta_d)(x,z,y,w)
=\left(\zeta_rx,\zeta_dz,\zeta_r^{-1}y,\zeta_d^{-1}w\right).$$
It is proven in \cite[\S3.3]{AIMgroup} that $\XX_d$ is birational to the quotient
$\CC_d\times\DD_d/\Delta$ via the map
$$(x,z,y,w)\mapsto \left(x=z^d,y=z^dxy,t=wz\right).$$

We are thus in
position to apply the machinery of Section~\ref{s:DPC}.  In
particular, it follows from Corollary~\ref{cor:dim-sha} that
\begin{equation}\label{eq:AIMdim}
\dim\sha(J_d)=
\dim\Hom_A\left(H^1(\CC_d)/p^n,H^1(\DD_d)/p^n\right)^\Delta
\end{equation}
for all sufficiently large $n$.
Subsection~\ref{ss:F-quotients} and
Proposition~\ref{prop:Fermat-cohom} describe the cohomology groups
$H^1(\CC_d)$ and $H^1(\DD_d)$ with their actions of Frobenius.  They
show in particular, that the dimension in the last display can be
computed by the methods of Section~\ref{s:p-adic}.

To spell this out, let 
$$I=J=\left\{(a,b)\in\Z/r\Z\times\Z/d\Z\left|
a\neq0,b\neq0,\<a/r\>+\<b/d\>\neq1\right.\right\},$$ 
let 
$$I_0=J_0=\left\{(a,b)\in\Z/r\Z\times\Z/d\Z\left|
a\neq0,b\neq0,\<a/r\>+\<b/d\>>1\right.\right\},$$ 
and let 
$$I_1=J_1=\left\{(a,b)\in\Z/r\Z\times\Z/d\Z\left|
a\neq0,b\neq0,\<a/r\>+\<b/d\><1\right.\right\}.$$ 
Section~\ref{s:Fermat} shows that the
crystalline cohomology groups $H^1(\CC_d)$ and $H^1(\DD_d)$ with their
action of Frobenius furnish
data ($M$, $N$, $I$, $J$, $c_i$, $d_j$) as in
Subsection~\ref{ss:data}, as well as the invariant $d(o)$ for each
orbit $o$ of $\<p\>$ on $I\times J$.  

Since $\Delta$ acts anti-diagonally, the orbits that contribute to the
right hand side of equation~\ref{eq:AIMdim} are those whose elements
$(i,j)=(a,b,a',b')$ satisfy $j=-i$, i.e., $a'=-a$ and $b'=-b$.  Write
$O^\Delta$ for the set of such orbits.  Applying
Theorem~\ref{thm:dim}, we conclude that
\begin{equation}\label{eq:AIM-sha-d}
\dim\sha(J_d)=\sum_{o\in O^\Delta}d(o).  
\end{equation}

We may identify the orbits in $O^\Delta$ with the orbits of $\<p\>$ on
$I$ via the projection $\pi_I:I\times J\to I$.  Also, since
$(i,-i)\in I_0\times J_1$ if and only if $i\in I_0$, and
$(i,-i)\in I_1\times J_0$ if and only if $i\in I_1$, we have
$$d(o)=\min(|\pi_I(o)\cap I_0|,|\pi_I(o)\cap I_1|).$$
We note that
$$|I_0|=|I_1|=\frac12\left((r-1)(d-1)-(\gcd(r,d)-1)\right),$$
which for fixed $r$ is asymptotic to $d(r-1)/2$ as $d\to\infty$.

Thus the sum on the right hand side of \eqref{eq:AIM-sha-d} becomes a
sum over orbits of $\<p\>$ on $I$, and the invariant $d(o)$ is
described ``on average'' in Section~\ref{s:equi}.  In particular, the
equidistribution result Proposition~\ref{prop:equi2} implies that
$$\dim\sha(J_d)=\sum_{o\in O^\Delta}d(o)=d(r-1)/2+\epsilon_d$$
where $\epsilon_d/d\to0$ as $d\to\infty$.

To finish the proof, we will show that $\tau(J_d)=O(H(J_d)^\epsilon)$
for all $\epsilon>0$ and that
$\deg(\omega_{J_d})\le d(r-1)/2+\epsilon_d$ where $\epsilon_d/d\to0$
as $d\to\infty$.  Once these claims are established,
Proposition~\ref{prop:sha-BS} implies that
$$\liminf_{d\to\infty}BS(J_d)
\ge\liminf_{d\to\infty}\frac{\dim\sha(J_d)}{\deg(\omega_{J_d})}\ge1.$$  
Taking into account the upper bound \eqref{eq:BS-HP} of Hindry and
Pacheco, we conclude that
$$\lim_{d\to\infty}BS(J_d)=1.$$

The assertion about $\tau(J_d)$ follows from the discussion of
Section~\ref{ss:tau-tame} and the fact (proven in
\cite[\S3.1]{AIMgroup}) that $X_d$ has semi-stable reduction at $t=0$
and $t=\infty$ whenever $r$ divides $d$.

It is proven in \cite[Proof of Proposition 7.5]{AIMgroup} that when
$r$ divides $d$, we have $\deg(\omega_{J_d})=d(r-1)/2$.   In general,
if $d'=\lcm(d,r)$, we have $\deg(\omega_{J_{d'}})=d'(r-1)/2$ and
Lemma~\ref{lemma:deg-omega} shows that
\begin{align*}
\deg(\omega_{J_d})&\le d(r-1)/2 + \frac{2(r-1)^2}{d'/d}\\
  &=d(r-1)/2+\epsilon_d.
\end{align*}
Since $d'/d$ is an integer, $\epsilon_d$ is bounded independently of $d$,
so $\epsilon_d/d\to0$ as $d\to\infty$.

This completes the proof of the theorem.
\end{proof}

\section{Quadratic twists of constant curves}
We conclude the paper with a study of Brauer-Siegel ratios of
quadratic twists of constant elliptic curves.
Throughout we let $p$ be an odd prime number, $\Fq$ a finite field of
characteristic $p$, and $K=\Fq(t)$.

\subsection{Twists of a constant supersingular curve}
Fix a supersingular elliptic curve $E_0$ over $\Fq$ and let
$E=E_0\times_\Fq K$.  For a positive integer $d$ relatively prime to
$p$, let $E_d$ be the twist of $E$ by the quadratic extension
$\Fq(t,\sqrt{t^d+1})$ of $K$. By results of Milne, the
Tate-Shafarevich group of $E_d$ is finite.

\begin{thm}
  We have
$$\lim_{\substack{d\to\infty\\(p,d)=1}}\BS(E_d)=1.$$
\end{thm}

\begin{proof}
  Let $\EE_d\to\P^1$ be the N\'eron model of $E_d/K$, and let $\CC_d$
  be the smooth projective curve over $\Fq$ defined by $y^2=x^d+1$ and
  equipped with the action of $\mu_2$ given by the hyperelliptic
  involution.  It is easy to see that $\EE_d$ is birational to the
  quotient of $\CC_d\times_\Fq E_0$ by the (anti-) diagonal action of
  $\mu_2$, i.e., by $\mu_2$ acting via the hyperelliptic involution on
  both factors.

We are thus in position to apply the machinery of Section~\ref{s:DPC}.
In particular, it follows from Corollary~\ref{cor:dim-sha} that
\begin{equation}\label{eq:Edim}
\dim\sha(E_d)=
\dim\Hom_A\left(H^1(\CC_d)/p^n,H^1(E_0)/p^n\right)^{\mu_2}
\end{equation}
for all sufficiently large $n$.

Subsection~\ref{ss:F-quotients} and
Proposition~\ref{prop:Fermat-cohom} describe the cohomology group
$H^1(\CC_d)$.  We recall the well-known description of $H^1(E_0)$: It
is a free $W$-module of rank 2 with a basis $e_0,e_1$ such that
$F(e_0)=d_0e_1$ and $F(e_1)=d_1e_0$ where $d_0$ is a unit of $W$ and
$d_1$ is $p$ times a unit.  (See \cite[\S5]{Dummigan95} for a detailed
account.)  To harmonize with earlier notation, let $J_0=\{0\}$,
$J_1=\{1\}$, and $J=J_0\cup J_1$, and equip $J$ with the non-trivial
action of $\<p\>$.

Also, let 
$$I=\Z/d\Z\setminus\{0,d/2\text{ (if $d$ is even)}\},$$ 
decomposed as $I_0=\{i|d/2<i<d\}$ and $I_1=\{i|0<i<d/2\}$.
Section~\ref{s:Fermat} and the preceding paragraph show that the
crystalline cohomology groups $H^1(\CC_d)$ and $H^1(E_0)$ with their
actions of Frobenius furnish data ($M$, $N$, $I$, $J$, $c_i$, $d_j$)
as in Subsection~\ref{ss:data}, as well as the invariant $d(o)$ for
each orbit $o$ of $\<p\>$ on $I\times J$.  We may thus compute the
dimension in the last display by the methods of
Section~\ref{s:p-adic}.

Since $\CC_d$ and $E_0$ are hyperelliptic, the $\mu_2$-invariant part
of their cohomology is trivial, so 
$$\Hom_A\left(H^1(\CC_d)/p^n,H^1(E_0)/p^n\right)^{\mu_2}
=\Hom_A\left(H^1(\CC_d)/p^n,H^1(E_0)/p^n\right).$$
Applying Theorem~\ref{thm:dim}, we conclude
that
\begin{equation}\label{eq:E-sha-d}
\dim\sha(E_d)=\sum_{o\in O}d(o)
\end{equation}
where the sum is over all orbits of $\<p\>$ on $I\times J$.

The equidistribution result Proposition~\ref{prop:equi3} implies that
$$\sum_{o\in O}d(o)=d/2+\epsilon_d$$
where $\epsilon_d/d\to0$ as $d\to\infty$.

Since $t^d+1$ has distinct roots, it is easy to see that
$\deg(\omega_{E_d})=\lceil d/2\rceil$.
Thus Corollary~\ref{cor:sha-BS} implies
that
$$\liminf_{d\to\infty}BS(E_d)
\ge\liminf_{d\to\infty}\frac{\dim\sha(E_d)}{\deg(\omega_{E_d})}=1.$$ 
Taking into account the upper bound \eqref{eq:BS-HP} of Hindry and
Pacheco, we conclude that
$$\lim_{d\to\infty}BS(E_d)=1.$$
\end{proof}

\subsection{Twists of an constant ordinary curve}
Now let $E_0$ be an ordinary elliptic curve over $\Fq$ and set
$E=E_0\times_\Fq K$.  One could use methods similar to those in the
last section to compute $\dim\sha(E_d)$ for the twist of $E$ by
$\Fq(t,\sqrt{t^d+1})$, but much more is easily deduced from results of
Katz in $p$-adic cohomology:

\begin{thm}\label{thm:BS-ord-twist}
  Let $E'$ be any quadratic twist of $E$.  Then
$$\dim\sha(E')=0.$$
\end{thm}

\begin{proof}
  A variety $X$ over a finite field is said to be \emph{Hodge-Witt} if
  all of its deRham-Witt cohomology groups $H^i(X,W\Omega^j_X)$ are
  finitely generated.  A curve is automatically Hodge-Witt, and a
  surface which satisfies the Tate conjecture is Hodge-Witt if and
  only if the dimension of its Brauer group (in the sense of
  Definition~\ref{prop:sha-dim}) is 0 \cite[\S1]{Milne75}.  In other
  words, a surface $X$ over $\Fq$ satisfying the Tate conjecture is
  Hodge-Witt if and only if
$$\lim_{n\to\infty}\frac{\log|H^2(X\times_\Fq\Fqn,\G_m)[p^\infty]|}{\log(q^n)}=0.$$

A theorem of Katz \cite{Katz83} says that a product of varieties is
Hodge-Witt if and only if one of the factors is ordinary and the other
is Hodge-Witt.

Now let $\CC\to\P^1$ be a double cover corresponding to a quadratic
extension $K'/K$.  Then the N\'eron model $\EE'\to\P^1$ of $E'/K$ is
birational to the quotient of $\CC\times_\Fq E_0$ by $\mu_2$ acting
diagonally by the hyperelliptic involutions.  Since $p>2$, the Brauer
group of the quotient is the $\mu_2$-invariant part of the Brauer
group of $\CC\times_\Fq E_0$, and the latter has dimension 0 since $E_0$
is ordinary.  It follows that the Brauer group of $\EE'$ has dimension
0 and so $\sha(E')$ has dimension zero.
\end{proof}

Thus for a quadratic twist of a constant, ordinary elliptic curve, our
$p$-adic methods do not give a non-trivial lower bound on the
Brauer-Siegel ratio.  This is compatible with Conjecture~1.7
of \cite{HindryPacheco16}, which predicts that the liminf of
$\BS(E')$ as $E'$ runs over all quadratic twists is 0.

We finish by remarking that Griffon has shown \cite{GriffonPoster}
that if $E_d$ is the twist of a constant ordinary $E/K$ by the
quadratic extension $\Fq(t,\sqrt{t^d+1})$, then as $d$ runs through
``supersingular'' integers, i.e., those that divide $p^f+1$ for some
$f$, the limit of $\BS(E_d)$ is 1.  In conjunction with
Theorem~\ref{thm:BS-ord-twist}, this shows that the Brauer-Siegel
ratio of an elliptic curve $E'$ may be large even when the dimension
of $\sha(E')$ is zero.

\bibliography{database}

\end{document}